\documentclass{amsart}
\usepackage{amsthm,amssymb,amsfonts,amsmath,mathrsfs}
\usepackage{dsfont}
\usepackage{graphicx}
\usepackage{enumerate}
\usepackage{hyperref}
\usepackage[nodvipsnames]{xcolor}
\usepackage{soul}

\newcommand{\notaremm}[1]{\textcolor{blue}{\ifmmode\text{\sout{\ensuremath{#1}}}
\else\sout{#1}\fi }}
\newtheorem{theorem}{Theorem}

\newtheorem{corollary}[theorem]{Corollary}

\newtheorem{definition}[theorem]{Definition}
\newtheorem{example}[theorem]{Example}

\newtheorem{lemma}[theorem]{Lemma}

\newtheorem{proposition}[theorem]{Proposition}
\newtheorem{remark}[theorem]{Remark}

\input{tcilatex}

\begin{document}
\title[Linear response for systems with noise]{Linear Response for dynamical
systems with additive noise}
\author{S. Galatolo}
\address{Dipartimento di Matematica, Universit\`a di Pisa, Largo Bruno
Pontecorvo 5, 56127 Pisa, Italy}
\email{stefano.galatolo@unipi.it}
\urladdr{http://pagine.dm.unipi.it/~a080288/}
\author{P. Giulietti}
\address{Centro di Ricerca Matematica Ennio De Giorgi, Scuola Normale
Superiore, Piazza dei Cavalieri 7, 56126 Pisa, Italy}
\email{paolo.giulietti@sns.it}
\date{\today }

\begin{abstract}
We show a linear response statement for fixed points of a family of Markov
operators which are perturbations of mixing and regularizing operators. We
apply the statement to random dynamical systems on the interval given by a
deterministic map $T$ with additive noise (distributed according to a
bounded variation kernel). We prove linear response for these systems, also
providing explicit formulas both for deterministic perturbations of the map $%
T$ and for changes in the noise kernel. The response holds with mild
assumptions on the system, allowing the map $T$ to have critical points,
contracting and expanding regions. We apply our theory to topological mixing
maps with additive noise, to a model of the Belozuv-Zhabotinsky chemical
reaction and to random rotations. In the final part of the paper we discuss
the linear request problem for these kind of systems, determining which
perturbations of $T$ produce a prescribed response.
\end{abstract}

\maketitle
\tableofcontents

\section{Introduction}

It is of major interest, both in pure mathematics and in applications, to
understand how the statistical properties of a physical system change when
it suffers from perturbations.

Sometimes small changes in the system produce small changes in its
statistical behavior, and the system is said to be statistically stable,
sometimes small changes lead to catastrophic events. {Beyond} the
qualitative statistical stability, {sometimes} it is {possible} to get a
quantitative understanding of the response of the system to small
perturbations, both in magnitude and direction. {This understanding can be
achieved in many cases of systems having \emph{linear response} to
perturbations.} As a matter of fact, the linear response of the system with
respect to a perturbation can be described by a suitable derivative, \emph{%
representing the rate of change of the relevant (physical, stationary)
invariant measure of the system with respect to the perturbation}. {Hence,}
the response describes the first order change of the equilibrium state
allowing us to get information about its robustness or sensitivity to change
in its parameters.

A concrete application of the above ideas can be found in modern geophysics:
if a climate model relies on a parameter which corresponds to external
forcing, it is relevant to study the stability of equilibrium states and
their evolution, informally ``the directions of change'' with respect to
such forcing. This includes information on the behavior of macroscopic
objects such as persistent atmospheric or oceanic currents, large vortices,
gyre and streams (see \cite{Mj,Lu} for more details or \cite{CGT} for an
inspirational review and further references). Furthermore, when considering
the management of a chaotic or complex system one is led to consider an
inverse problem related to the linear response: is it possible to realize a
specific change of the equilibrium state by controlling the perturbation? Is
there an \textquotedblleft optimal\textquotedblright\ way of doing so? (See 
\cite{GP, Kl16, ADF, Mac} and Section \ref{sec:control} for a more detailed
introduction to this control problem). We will investigate this family of
questions specifically in the case of a random dynamical system given by a
deterministic map with additive noise.

{While the analysis of the linear response of a system may provide a lot of
information, it is a fact that not all systems are well behaved with respect
to perturbations. In particular} not all dynamical systems are even
``statistically stable'' \footnote{%
Informally speaking, a system is said to be statistically stable under a
certain kind of perturbation if the invariant measure of interest varies
continuously under the perturbation.}. The identity map is an example of a
system which is not stable under any reasonable formal definition of
statistical stability. {On the other hand}, many systems of applied and
theoretical interest are stable. The literature about qualitative
statistical stability in dynamical systems is vast, see the introduction of 
\cite{AV02} for an historical account on the notion of statistical
stability, see \cite{ACF10, ASsu} and references therein for more recent
results. Quantitative estimates show the existence of systems which are
statistically stable but have a response to perturbation of orders of
magnitude larger than the perturbation {itself}. The stationary measure may
vary just with H\"older modulus of continuity, or even worse, see e.g. \cite%
{BBS, D2, Gmann, Gpre, zz,GLMK,KLS15}. %
%\notaadd{ Moreover, note that even 
%it is very tempting to look at $C^k$ regularity  for the ``derivative density'', 
%even in simpler (and nicer situations) given a family of transfer operator the 
%``curve'' of corresponding densities is somehow rough  (see \cite[Theorem E, 
%Section 6.2]{GLMK} and \cite[Corollary 1.3]{KLS15}) and best described by a 
%Wasserstein type of norm. Since such norm can capture the diffusive regularizing 
%behaviour of the randomness in a very natural way, we hope in future 
%collaborations to investigate the relationship between the type of randomness 
%proposed in \cite{BRS} and such norms. } 

It is worth to remark that a general relation holds, between this modulus of
continuity and the speed of convergence to equilibrium of the system (see 
\cite{Gpre,Gmann} see also \cite{Mit} for earlier results limited to Markov
chains).

The results discussed so far mostly regard the stability of deterministic
systems under deterministic perturbation. In fact, the statistical stability
of deterministic systems under \textit{small} random perturbations is called 
\textit{stochastic stability} and it has been studied quite extensively.
Early results on piecewise monotonic transformation with some regularity
assumption on both the transformation and the invariant measures were
already obtained by Keller \cite{Ke82}. The book \cite{Ki88} provides an
excellent starting point for the study of the subject.

Concerning systems having Linear Response, several results have been proved
for deterministic perturbations of deterministic dynamical systems. Starting
with Ruelle, it is known that in smooth, uniformly hyperbolic systems the
physical invariant measure changes smoothly and a formula for such
differentiation can be obtained (see \cite{R}). Similar results can be
proved in some non uniformly expanding or hyperbolic cases (see e.g. \cite%
{Ba1, BaSma,BKL, D, BahSau,BT,K,zz}). On the other hand, it is known that
the linear response does not always hold, due to the lack of regularity of
the system, of the perturbation or the lack of sufficient hyperbolicity (see 
\cite{Ba1, BBS, zz, Gpre}). An example of linear response for small random
perturbations of deterministic systems (which uses ideas similar to what we
will implement here) is produced in \cite{Li2}. The survey \cite{BB} has
exhaustive list of classical references on the subject at hand. The reader
interested in seeing linear response from a point of view closer to physics
can consult \cite{Luu, A}. Rigorous numerical approaches for the computation
of the linear response are available, to some extent, both for deterministic
and random systems (see \cite{BGNN, PV}).

About linear response of random dynamical systems in general, less is known. 
{As it is common in literature, by random system we mean a system defined by
a random dynamics which might, or might not, involve a deterministic part
e.g. a dynamics defined only by a transition operator acting on probability
distributions versus a random choice between deterministic maps. To study
the response of such systems we may perturb any of its defining components }
and look at the change in the resulting statistical properties. Results for
this kind of systems were proved in \cite{HM}, {where} {the} technical
framework {was} adapted to stochastic differential equations. {Moreover,}
there is a recent work \cite{BRS} proving linear response for a class of
random uniformly and non uniformly expanding systems. See Remark \ref%
{rem:comparison} for a detailed comparison between the present work, and
these two closely related results.

In the random case, like in the deterministic case, a fruitful strategy to
study the stability of a system, relies on noticing that the stationary
measures of interest are fixed points of the transfer operators associated
to the system we consider; thus, linear response statements or quantitative
stability results can be proved by first proving perturbation theorems for
suitable operators as done in \cite{HM, notes, JS, Gpre, KL, Li2}.

In this paper we apply this general strategy to an important class of random
dynamical systems. We prove a general fixed point stability result and
linear response for a class of Markov operators satisfying mild assumptions
adapted to random dynamical systems where the unperturbed operators have both%
\emph{\ mixing and regularizing} properties. We then apply this statement to 
\textit{systems with additive noise} i.e. systems where the dynamics map a
point deterministically to another point and then some random perturbation
is added independently at each iteration, according to a certain bounded
variation noise distribution kernel. {Starting with operators of this kind,
we construct} a family of transfer operators obtained {perturbatively}
either by changing the deterministic part of the system or by modifying
continuously the shape of the noise.

The main novelty and focus of this work is the {exploitation of} the
regularizing effect of the additive noise (essentially provided by the
regularizing effect of convolutions). This allows to work with almost no
assumptions on the deterministic part of the dynamics. {The few assumptions
we require are easy to be verified in many systems, they guarantee linear
response and yield simple explicit formulas}. {Our framework} allows to
treat systems in which the deterministic part of the dynamics is not
hyperbolic (or cannot be reduced to some hyperbolic dynamics by a suitable
renormalization or acceleration). In some sense the regularizing effect of
the noise, on suitable functional spaces, in our approach plays the role of
the Lasota-Yorke-Doeblin-Fortet inequalities, as commonly used in many other
functional analytic approaches to the study of the statistical properties of
systems. {Let us stress that a good feature of } our technical {solution} is
that the mixing and regularization assumptions are required only for the {%
unperturbed system} (something similar also appear in \cite{HM,JS}), this
allows us to consider a wide class of perturbations, without requiring
uniform estimates {with respect to} the perturbation.

We show the flexibility of our approach by applying it to nontrivial systems
of different kinds, as we will see in Section \ref{sec:finalmain}. We remark
that the presence of noise is natural in applications (see e.g. \cite%
{MT,XX,Gh, CGS,rdapp} for models in several applied contexts) and from the
mathematical point of view, this simplifies the functional analytic study of
the system allowing more regularity and robustness.

\noindent\textbf{Outline of the paper.} The paper has the following
structure: In section \ref{sec:Main} we introduce the relevant objects and
show the main results of the paper; in section \ref{general} we prove the
statements regarding our general Markov operator, in section \ref%
{sec:derivative} we consider systems with additive noise and prove the
required properties, computing the form for the ``derivative'' operators. In
section \ref{sec:finalmain}, we apply our theory to some nontrivial
examples. Last, in section \ref{sec:control} we consider the ``linear
request'' control problem associated to the linear response statement.

\bigskip\noindent {\ \textbf{Acknowledgments.} P.G. thanks the Universidade
Federal do Rio Grande do Sul, Porto Alegre, Brazil, where part of the work
was done. P.G. has been supported in part by EU Marie-Curie IRSES
Brazilian-European partnership in Dynamical Systems (FP7-PEOPLE-2012-IRSES
318999 BREUDS). P.G. acknowledges the support of the Centro di Ricerca
Matematica Ennio de Giorgi and of UniCredit Bank R\&D group for financial
support through the ``Dynamics and Information Theory Institute'' at the
Scuola Normale Superiore. S.G. thanks GNAMPA-INdAM for partial support
during this research. P.G. thanks B. Kloeckner, F. Flandoli, W. Bahsoun, M.
Monge for useful discussions on the matter along the years.}

\section{Main results\label{sec:Main}}

\subsection{A general Linear Response result for regularizing Markov
operators}

In Section \ref{general} we prove a linear response theorem for the fixed
points of a family of Markov operators. The structure of the statement is
not far from the one of \cite{HM} or other quantitative stability theorems,
but the technical solutions have been chosen here having in mind the
applications to systems with additive noise presented in the following
sections. We will consider a family of Markov operators and their action on
several spaces of {various regularity}. Let us introduce these spaces: let
us consider $BS[0,1]$ the set of Borel finite measures with sign on $[0,1]$;
let us also consider the space of finite absolutely continuous measures $%
L^{1}[0,1]$ equipped with the usual $\Vert \cdot \Vert _{1}$ norm as a
subset of $BS[0,1]$. Dealing with measures in $L^{1}[0,1]$ and related
densities (let us denote $L^{1}[0,1]$ as $L^{1}$ for short when the
underlying space is clear from the context, in other cases the argument of $%
L^{1}[\cdot ]$ will be specified) we will often identify the measure with
the density, when it is possible to do so, and profit from the expressive
power of the integral notation, writing for example $\int_{[0,1]}f~dm=0$,
where $m$ stand for the Lebesgue measure, {or} $f([0,1])=0.$ It will be
useful to also consider the set of even more regular measures having bounded
variation density, whose definition we are going to recall.

\begin{definition}
\label{def:spaces} Let $f:[0,1]\rightarrow \mathbb{R}$ and $%
P=P(x_{0},x_{1},\ldots ,x_{n})$ a partition of $[0,1]$. 
\begin{equation}
Var_{[0,1],P}(f):=\sum_{k=1}^{n}|f(x_{k})-f(x_{k-1})|
\end{equation}%
If there exists $M$ such that $Var_{[0,1]}(f):=\sup_{P}Var_{[0,1],P}\leq M$
then $f$ is said to be of Bounded Variation. Let be the Banach space of
Borel measures having a bounded variation density be denoted as 
\begin{equation*}
BV[0,1]=\{\mu \in L^{1},Var_{[0,1]}(\frac{d\mu }{dm})<\infty \}.
\end{equation*}%
with the norm $\Vert \mu \Vert _{BV}=\Vert \mu \Vert _{1}+Var_{I}(\frac{d\mu 
}{dm})$. We will always use $BV$ for $BV[0,1]$, unless $BV[\cdot ]$
specifies an argument for the space.
\end{definition}

Let us consider a normed vector space $(B_{w},\|\cdot \|_{w})$, with $%
BS\supseteq B_{w}\supseteq L^{1}$ and $\|\cdot \|_{w}\leq \|\cdot \|_{1}$ on 
$B_{w} \cap L^{1} $. We also need to consider spaces of zero average
measures.

\begin{definition}
\label{v}Let us define the space of zero average measures $V\subset L^{1}$ as%
\begin{equation}
V:=\{f\in L^{1}[0,1]~s.t.~f([0,1])=0\}  \label{d1}
\end{equation}%
and 
\begin{equation}
V_{w}:=\{\mu \in B_{w}~s.t.~~\mu ([0,1])=0\}.  \label{d2}
\end{equation}
\end{definition}

Let us hence consider a family of Markov operators $L_{\delta
}:BS([0,1])\rightarrow BS([0,1])$, where $\delta \in \lbrack 0,\overline{%
\delta })$. The family of operators $L_{\delta }$ will be considered also as
acting on $L^{1}$ and $BV$, the space of measures having bounded variation
density. \ With a small abuse of notation we will use $L_{\delta }$ to
indicate the operators acting on these different spaces without changing the
notation. Recall that a Markov operator $L$ is positive and preserves
probability measures: if $f\geq 0$ then $Lf\geq 0$ and $\int Lfdm=\int fdm$.
Let us denote by ${\mathds{1}}$ the identity operator. Let us denote by $%
R(z,L)$ the resolvent related to an operator $L$, formally defined as 
\begin{equation}
R(z,L)=\sum_{n=0}^{\infty }z^{n}L^{n}=(z{\mathds{1}}-L)^{-1}  \label{d3}
\end{equation}%
{wherever} the infinite series converges. If we suppose that each operator $%
L_{\delta }$ has a fixed probability measure in $BV[0,1]$, we show that
under mild further assumptions these fixed points vary smoothly in the
weaker norm $\Vert \cdot \Vert _{w}$. The following is a general Linear
Response statement for regularizing operators which we will use in several
examples of random systems and perturbations.

\begin{theorem}
\label{th:linearresponse} Suppose that the family of operators $L_{\delta}$
satisfies the following:

\begin{itemize}
\item[(LR0)] $f_{\delta }\in BV[0,1]$ is a probability measure such that $%
L_{\delta }f_{\delta }=f_{\delta }$ for each $\delta \in \left[ 0,\overline{%
\delta }\right) $. Moreover there is $M\geq 0$ such that $\Vert f_{\delta
}\Vert _{BV}\leq M$ for each $\delta \in \left[0,\overline{\delta }\right)$.

\item[(LR1)] (mixing for the unperturbed operator) For each $g\in BV[0,1]$
with ${\int_{I}g\,dm=0}$ 
\begin{equation*}
\lim_{n\rightarrow \infty } \|L_{0}^{n}g\|_{1}=0;
\end{equation*}

\item[(LR2)] (regularization of the unperturbed operator) $L_{0}$ is
regularizing from $B_{w}$ to $L^{1}$ and from $L^{1}$ to Bounded Variation
i.e. $L_{0}:(B_{w},\Vert \cdot \Vert _{w})\rightarrow L^{1}$ , ${L_{0}:L^{1}
\rightarrow BV[0,1]}$ are continuous.

\item[(LR3)] (small perturbation and derivative operator) There is $K\geq 0$
such that\footnote{%
We find the following notations for operators norms very useful. If $A,B$
are two normed vector spaces and $T:A \to B$ we write $\| T \|_{A \to B} :=
\sup_{f \in A, \| f \|_A \leq 1} \| Tf \|_{B}$} $\left\vert |L_{0}-L_{\delta
}|\right\vert _{L^{1}\rightarrow (B_{w},\Vert \cdot \Vert _{w})}\leq K\delta
,$ and $\left\vert |L_{0}-L_{\delta }|\right\vert _{BV\rightarrow V}\leq
K\delta $. There is ${\dot{L}f_{0}\in V_{w}}$ such that%
\begin{equation}  \label{derivativeoperator}
\underset{\delta \rightarrow 0}{\lim }\left\Vert \frac{(L_{0}-L_{\delta })}{%
\delta }f_{0}-\dot{L}f_{0}\right\Vert _{w}=0.
\end{equation}
\end{itemize}

Then $R(z,L_{0}):V_{w}\rightarrow $ $V_{w}$ is a continuous operator and we
have the following Linear Response formula 
\begin{equation}
\lim_{\delta \rightarrow 0}\left\Vert \frac{f_{\delta }-f_{0}}{\delta }%
-R(1,L_{0})\dot{L}f_{0}\right\Vert _{w}=0.  \label{linresp}
\end{equation}%
Thus $R(1,L_{0})\dot{L}f_{0}$ represents the first order term in the change
of equilibrium measure for the family of systems $L_{\delta }$.\newline
\end{theorem}

In the following Lemma we see that the assumptions $(LR1,...,LR3)$ of
Theorem \ref{th:linearresponse} are sufficient to establish the existence of
a unique fixed probability measure $f_{0}$ of $L_{0}$ in BV.

\begin{lemma}
\label{extuniq} Under the above assumptions $(LR1,...,LR3)$ the system $%
L_{0} $ has a unique fixed probability measure $f_{0}\in $ BV.
\end{lemma}

\begin{proof}
First let us show the existence of a positive invariant measure with Bounded
Variation density. Let us consider the iterates $L_{0}^{n}(m)$ of the
Lebesgue measure $m$. Because of the regularization property all of these
measures have a Bounded Variation density. Since $L_{0}$ is a Markov
operator, $\Vert L_{0}^{n}(m)\Vert _{1}\leq 1$ for all $n$. Since $%
L_{0}:L^{1}\rightarrow BV[0,1]$ is continuous we have that 
\begin{equation*}
\Vert L_{0}^{n}(m)\Vert _{BV}\leq \Vert L_{0}\Vert _{L^{1}\rightarrow
BV}\Vert L_{0}^{n-1}(m)\Vert _{1}\leq \Vert L_{0}\Vert _{L^{1}\rightarrow BV}
\end{equation*}%
for every $n$. Now, like in the classical Krylov Bogoliubov argument let us
consider the Ces\`{a}ro averages $m_{n}=\frac{1}{n}\sum_{0\leq i\leq
n}L^{i}(m)$. This is a sequence having uniformly bounded variation. Now
applying the classical Helly selection theorem to $m_{n}$ we get a
converging subsequence to some limit measure $f_{0}$ in $L^{1}$. As in the
classical argument, $f_{0}\in L^{1}$ will be invariant. But since $L_{0}$ is
regularizing $L_{0}(f_{0})=f_{0}\in BV$. For the uniqueness, if $f,g$ are
two {such} fixed points we have that $f,g\in BV$ and $f-g\in V$ {i.e. $f-g$
has zero average}. Thus, by the mixing assumption, 
\begin{equation*}
\Vert f-g\Vert _{1}=\Vert L_{0}^{n}(f)-L_{0}^{n}(g)\Vert _{1}=\Vert
L_{0}^{n}(f-g)\Vert _{1}\rightarrow 0
\end{equation*}%
contradicting $f\neq g$ in $L^{1}$.
\end{proof}

\begin{remark}
The assumption $(LR0)$ is naturally verified for a family of uniformly
regularizing operators. We will show that the assumption holds for general
systems with additive noise in Lemma \ref{lem:fixedpoint}. We remark that
the assumption \ can be replaced by a weaker assumption: $\Vert f_{\delta
}\Vert _{BV}\leq o(\delta ^{-1})$ as $\delta \rightarrow 0$ (see Remark \ref%
{weakLR0}).
\end{remark}

\begin{remark}
\label{rmkLM}The mixing assumption in Item $(LR1)$ is required only for the 
\emph{\ unperturbed operator} $L_{0}$. This {\ requirement } is {somehow
expected by} systems with additive noise having some sort of
indecomposability or topological mixing, we will verify it in several
examples in Section \ref{sec:finalmain} of different nature, also using a
computer aided proof. Furthermore, the assumption is satisfied, for example,
if there is an iterate of the transfer operator having a strictly positive
kernel, see Corollary 5.7.1 of \cite{LM}.
\end{remark}

\begin{remark}
The regularization property $(LR2)$ is also required only for the
unperturbed operator. We remark that the regularization assumption in the
class of systems with additive noise we aim to consider is granted by the
presence of noise. This is shown under some minimal set of assumptions on
the\ deterministic part of the dynamics in Corollary \ref{regcor}.
\end{remark}

\begin{remark}
\label{rrm6}As stated at Item $(LR3)$, $\dot{L}f_{0}$ is an element of $%
V_{w} $. Depending on the kind of system and perturbation considered, it
could be natural to consider $\dot{L}f_{0}$ in a space of distributions
(completing signed measures with respect to limits in the $\Vert \cdot \Vert
_{w}$ norm), however for the purpose of this paper, signed measures are
sufficient. We remark that one can think of%
\begin{equation}
\dot{L}:=\lim_{\delta \rightarrow 0}\frac{(L_{0}-L_{\delta })}{\delta }
\label{de}
\end{equation}%
as a {\textquotedblleft }general\textquotedblright\ derivative operator.
This limit may converge in various topologies, depending on the system and
the perturbations considered, giving different linear response statements
(as we will see in the following Proposition \ref{derivop}, Corollary \ref%
{finalcor} and the examples in Section \ref{sec:finalmain}). This is why it
is worth considering a general norm $\Vert \cdot \Vert _{w}$ in our
framework. Note that it is often handy to have some explicit
characterization of the derivative operator, thus obtaining precise
information on the structure of the linear response (see equation \eqref{linresp}). Explicit expressions for the derivative operator in interesting cases
will be presented in Proposition \ref{derivop}.
\end{remark}

\begin{remark}
In the above statement, $L^{1}[0,1]$ and $BV[0,1]$ play the role of weak and
strong space for which there is a compact immersion, {such that} $L_{0}^{n}$
is uniformly bounded as an operator on the weak space and for which $L_{0}$
is regularizing from the weak space to the strong one (see the proof in
Section \ref{general}). The statement can be generalized considering other
spaces with the same properties. For the sake of clarity and for the kind of
examples we are going to consider in the paper we decided to state our
results in this framework.
\end{remark}

\begin{remark}
In the above statement the weak norm $\Vert \cdot \Vert _{w}$ could be the $%
L^{1}$ norm itself. For an example of a weak norm $\Vert \cdot \Vert _{w}$
strictly weaker than $L^{1}$ which is used in the next sections let us
consider the Wasserstein-Kantorovich norm defined on $BS$ as%
\begin{equation}
\Vert \mu \Vert _{W}=\underset{\Vert g\Vert _{\infty }\leq 1,~Lip(g)\leq 1}{%
\sup }\int_{0}^{1}g(x)d\mu .  \label{Wass}
\end{equation}%
Where $Lip(g)$ is the best Lipschitz constant of $g$. We remark that $BS$ is
not complete with this norm. The completion leads to a distributions space
which is the dual of the space of Lipschitz functions. {\ Note that $\Vert
\cdot \Vert _{W}\leq \Vert \cdot \Vert _{1}\leq \Vert \cdot \Vert _{BV}$. 
%\notaremm{ Using the regularizing properties of the convolution (see Lemma 
%\ref{convoocopy1})  and Helly's selection principle, by the compact immersion of 
%$BV$ in $L^{1}$ it is easy to see that for the operators  $L_{\xi ,T_{\delta }} 
%$ there are probability measures $f_{\xi ,\delta }\in BV$ such that $ L_{\xi 
%,T_{\delta }}f_{\xi ,\delta }=f_{\xi ,\delta }$.}
}
\end{remark}

\subsection{Application to systems with additive noise\label{adnois}}

In Section \ref{sec:derivative} we focus our attention on systems with
additive noise and use Theorem \ref{th:linearresponse} to get linear
response statements for relevant perturbations of such systems. Let $%
T_{\delta }:[0,1]\rightarrow \lbrack 0,1]$ be a family of Borel, nonsingular
maps parametrized by $\delta \in \lbrack 0,\overline{\delta }]$ which are
\textquotedblleft small"\ perturbations of $T_{0}$ in a sense which will be
specified later (see Proposition \ref{derivop}). We consider the composition
of $T_{\delta }$ and a random additive perturbation, distributed according
to a probability density $\rho _{\delta }$, where $\rho _{\delta }\in BV$ is
a family of noise kernels with support in $[0,1]$ which are considered as
small perturbations of some initial noise kernel $\rho _{0}$. We consider
the response of the system to small \textquotedblleft reasonable"\ changes
of $T_{0}$ or $\rho _{0}$. More precisely, a random dynamical system with
additive noise on $[0,1]$ and reflecting boundary conditions in our context
is a random perturbation of a deterministic map, defined {inductively as the
process 
\begin{equation}  \label{systm}
x_{n+1} = T(x_{n})\hat{+}\Omega_{n}
\end{equation}
} %notaremm{  ggg $  \label{systm} x\rightarrow T(x)\hat{+}\Omega (n) $ }
where $T:[0,1]\rightarrow \lbrack 0,1]$ is a Borel measurable map, $\Omega_n$
is an i.i.d. process distributed according to a probability density $\rho $ {%
\ $\in BV$} and $\hat{+}$ is the \textquotedblleft reflecting boundaries
sum" on $[0,1]$ which takes care of the boundary effects, sending back to
the interval points sent outside by the noise, formally defined as follows.

\begin{definition}
Let $\pi :{\mathbb{R}}\rightarrow [0,1]$ be the piecewise linear map 
\begin{equation}
\pi (x)=\min_{i\in \mathbb{Z}}|x-2i|.  \label{ppi}
\end{equation}%
Let $a,b\in {\mathbb{R}}$ then 
\begin{equation*}
a\hat{+}b:=\pi (a+b)
\end{equation*}
where $+$ is the usual sum operator on $\mathbb{R}$. By this $a\hat{+}b\in
[0,1].$
\end{definition}

To approach the question, we will consider the transfer operators associated
to this kind of systems (see e.g. \cite[Chapter 5]{Viana} or {\ \cite[%
Section 10]{LM94} } for basic notions about transfer operators for random
systems). For each {$\delta \in \lbrack 0,\overline{\delta })$ we consider
the transfer operator $L_{{\delta }}:BS\rightarrow BS$ associated to the
system with deterministic part $T_{\delta }$ and noise distributed as $\rho
_{\delta }$. This will be the composition of the transfer operator
associated to the deterministic part of the dynamics and the one associated
to the action of the noise. Recall that the transfer operator associated to
a deterministic transformation $T_{\delta }$ is defined, as usual by the
pushforward map also denoted by $(T_{\delta })_{\ast },$ by 
\begin{equation}
\lbrack L_{T_{\delta }}(\mu )](A):=[(T_{\delta })_{\ast }\mu ](A):=\mu
(T_{\delta }^{-1}(A))  \label{pish}
\end{equation}%
for each measure with sign $\mu $ and Borel set $A$. When $T$ is nonsingular 
$L_{T}$ preserves absolutely continuous finite measures and can be
considered as an operator $L^{1}\rightarrow L^{1}$. We consider as the
(annealed) transfer operator associated to the system with additive noise
the operator $L_{\xi ,T_{\delta }}:BS\rightarrow L^{1}$ defined as%
\begin{equation}
L_{\xi ,T_{\delta }}(f):=\rho _{\xi }\hat{\ast}L_{T_{\delta }}(f)
\label{eq:transop}
\end{equation}%
where $\hat{\ast}$ stands for an operator which is a \textquotedblleft
boundary reflecting" convolution taking care of the boundary effects,
defined as follows. }

\begin{definition}
\label{def:reflectingboundary} Let $\mu \in BS({\mathbb{R}})$. Let $\pi :{%
\mathbb{R}}\rightarrow \lbrack 0,1]$ be the piecewise linear map defined in $%
($\ref{ppi}$)$, and $\pi _{\ast }:BS[{\mathbb{R}}]\rightarrow BS[0,1]$ its
associated pushforward map (see \eqref{pish}). We consider $\pi _{\ast }\mu
\in BS[0,1]$ as the \textquotedblleft reflecting boundary\textquotedblright\
version of $\mu $ on the interval.
\end{definition}

\begin{definition}
\label{def:hatconv} Let $f\in BS$, $\rho _{\xi }\in BV[0,1]$. Let $\hat{f}%
\in BS[\mathbb{R}]$ defined by $\hat{f}=1_{[0,1]}f$ and $\hat{\rho}_{\xi
}\in L^{1}[\mathbb{R}]$ by $\hat{\rho}_{\xi }=1_{[0,1]}\rho _{\xi }$ (where $%
1_{A}$ stands for the indicator function of the set $A$). We define 
\begin{equation}
\rho _{\xi }\hat{\ast}f=\pi _{\ast }(\hat{\rho}_{\xi }\ast \hat{f})
\label{hat}
\end{equation}%
where $\ast $ stands for the usual convolution operator and $\pi _{\ast }$
acts on the associated measure. {Note that $\rho _{\xi }\hat{\ast}f\in
BS[0,1]$.}
\end{definition}

More details on the convolution of measures and densities will be given in
Section \ref{sec:derivative}, where we will also prove several
regularization properies of $\hat{\ast}$.

In the following we will consider two possible choices for the weak norm $%
||~||_{w}$. \ One is the usual $L^{1}$ norm and the other is the
Wasserstein-like norm defined in $($\ref{Wass}$)$. This norm will be useful
when considering perturbations of noise kernels having discontinuities, as
the uniformly distributed noise{\ (see }Section \ref{sec:derivative} and in
particular Example \ref{exnoise}).

As we will see, these transfer operators satisfy the assumptions of Theorem %
\ref{th:linearresponse} \emph{provided they are mixing}{, }$T_{0}$ has
suitable, mild, regularity properties{\ and we consider suitable
perturbations}. Indeed, we will see at Lemma \ref{lem:fixedpoint} that such
operators have a fixed probability measure $f_{\xi ,\delta }$ with bounded
variation, and the variation of $f_{\xi ,\delta }$ is bounded by the
variation of $\rho _{\xi }$ (assumption $(LR0)$). The regularizing effect
asked in assumption $(LR2)$ for the unperturbed transfer operator is
provided by the effect of the convolution (see Corollary \ref{regcor}). Now
let us consider assumption $(LR3)$ and study the derivative operator
associated to some important class of perturbations. We consider
perturbations of $T_{0}$ or $\rho _{0}$ such that $L_{\xi ,T_{\delta }}$ 
%\notaadd{have fixed points $f_{\xi ,\delta }$ (see Lemma \ref{lem:fixedpoint}), 
%they are mixing, such that $L_{\xi ,T_{\delta }}$ }
is a small perturbation of $L_{0,T_{0}}$ in a suitable sense; and such that
the derivative operator $\dot{L}$ is well defined. Under certain
assumptions, we can compute a formula for $\dot{L}$. In fact we have the
following (see Propositions \ref{pertmap}, \ref{plusnoise}, \ref{nearness
copy(1)}).

\begin{proposition}
\label{derivop} Let $L_{0,T_{0}}$ be the transfer operator of a system with
additive noise, as defined above. For perturbations of the map $T_{0}$ or of
the noise kernel $\rho _{0}$ the following hold:

\begin{description}
\item[a] (perturbing the map) Let $D_{\delta }:[0,1]\rightarrow \lbrack 0,1]$
be a bilipschitz homeomorphism near the identity such that $D_{\delta }={%
\mathds{1}}+\delta S$ , and $S$ is a $1$-Lipschitz map\footnote{%
By $1$-Lipschitz we mean a function $S:{\mathbb{R}}\rightarrow {\mathbb{R}}$
such that $|S(x)-S(y)|\leq |x-y|$.} with $S(0)=S(1)=0$. Denote by $\chi _{S}$
the support of $S,$ suppose $\chi _{S}$ is a finite union of intervals.
Suppose $L_{T_{0}}:BV([0,1])\rightarrow BV(\chi _{S})$ is continuous. 
\footnote{%
There is $K\geq 0$ such that $\Vert 1_{\chi _{S}}\cdot L_{T_{0}}f\Vert
_{BV}\leq K\Vert f\Vert _{BV}$ for each $f\in BV[0,1]$. This condition is
sufficient to ensure that $L_{T_{0}}(f)S$ is a $BV$ function and then $%
(L_{T_{0}}(f)S)^{\prime }\in BS.$} Let $T_{\delta }=D_{\delta }\circ T_{0}$.
Then {$\| L_{0,T_{\delta }} f - L_{0,T_{0}}f\|_{1} \leq \delta \mathrm{Var}%
(\rho_{0})\|f\|_{1}$} and for all $f\in BV,$ the following limit (defining
the derivative operator) converges in $L^{1}${\ 
\begin{equation}
\underset{\delta \rightarrow 0}{\lim }\left\Vert \frac{(L_{0,T_{\delta
}}-L_{0,T_{0}})}{\delta }f-\rho _{0}\hat{\ast}(-L_{T_{0}}(f)S)^{\prime
}\right\Vert _{1}=0.  \label{eq:noiseresponse1}
\end{equation}
} In this formula, $(L_{T}(f)S)^{\prime }$ should be interpreted as a
measure.

\item[b] (perturbing the noise) Suppose $\Vert \rho _{\xi }-\rho _{0}\Vert
_{1}\leq C\xi $, $f\in L^{1}$, then 
\begin{equation*}
\Vert L_{\xi ,T_{0}}f-L_{0,T_{0}}f\Vert _{W}\leq \Vert L_{\xi
,T_{0}}f-L_{0,T_{0}}f\Vert _{1}\leq C\xi \Vert f\Vert _{1}.
\end{equation*}

Suppose there exists $\dot{\rho}\in BS$ such that 
\begin{equation*}
\underset{\xi \rightarrow 0}{\lim }\left\Vert \frac{\rho _{\xi }-\rho _{0}}{%
\xi }-\dot{\rho}\right\Vert _{W}=0
\end{equation*}%
(the convergence is with respect to the Wasserstein-Kantorovich norm defined
at (\ref{Wass})) suppose $T_{0}$ is nonsingular and $f\in L^{1}$, then 
\begin{equation}
\underset{\xi \rightarrow 0}{\lim }\left\Vert \frac{(L_{\xi
,T_{0}}-L_{0,T_{0}})}{\xi }f-\dot{\rho}\hat{\ast}L_{T}(f)\right\Vert _{W}=0.
\label{13.5}
\end{equation}
\end{description}
\end{proposition}

Applying the findings of Proposition \ref{derivop}, Lemma \ref%
{lem:fixedpoint} \ and Corollary \ref{regcor} \ to Theorem \ref%
{th:linearresponse} (see \ Section \ref{sec:derivative} for more details) we
get the following linear response results for perturbations of systems with
additive noise

\begin{corollary}
\label{finalcor}Let $L_{0,T_{0}}$ be the transfer operator of a system with
additive noise, as above. Suppose that $L_{0,T_{0}}$ satisfies assumption $%
(LR1)$ (mixing).\ For perturbations of $\ T_{0}$ or of the noise kernel $%
\rho _{0}$ as in Proposition \ref{derivop} the following hold:

\begin{description}
\item[a] (perturbing the map) Let $T_{\delta }$ be a perturbation of $T_{0}$
as defined at Item a) of Proposition \ref{derivop} and let $L_{0,T_{\delta
}} $ be the transfer operator associated to the system with additive noise
whose deterministic part is $T_{\delta }$ \ and the noise is distributed
according to $\rho _{0}\in BV.$ Let $f_{\delta }$ be fixed probability
measures of the transfer operators $L_{0,T_{\delta }},$ i.e. $L_{0,T_{\delta
}}f_{\delta }=f_{\delta }$ then $R(z,L_{0,T_{0}}):V\rightarrow $ $V$ is a
continuous operator and we have the following Linear Response formula 
\begin{equation}
\lim_{\delta \rightarrow 0}\left\Vert \frac{f_{\delta }-f_{0}}{\delta }%
-R(1,L_{0})\rho _{0}\hat{\ast}(-L_{T_{0}}(f)S)^{\prime }\right\Vert _{1}=0.
\label{finalres1}
\end{equation}

\item[b] (perturbing the noise) Suppose $T_{0}$ is Lipschitz, $\Vert \rho
_{\xi }-\rho _{0}\Vert _{1}\leq C\xi $, and that there exists $\dot{\rho}\in
BS$ such that 
\begin{equation*}
\underset{\xi \rightarrow 0}{\lim }\left\Vert \frac{\rho _{\xi }-\rho _{0}}{%
\xi }-\dot{\rho}\right\Vert _{W}=0
\end{equation*}%
then $R(z,L_{0,T_{0}}):V_{W}\rightarrow $ $V_{W}$ is a continuous operator
and we have the following Linear Response formula 
\begin{equation}
\lim_{\delta \rightarrow 0}\left\Vert \frac{f_{\delta }-f_{0}}{\delta }%
-R(1,L_{0})\dot{\rho}\hat{\ast}L_{T}(f)\right\Vert _{W}=0.  \label{finalres2}
\end{equation}
\end{description}
\end{corollary}

\begin{remark}
\label{rem:comparison} We now would like to draw {the comparison between the
above results, which are the main tools we are using in our applications,
and the most similar results available in literature}. 
% \notarem{the attention to the  differences between the present situation and the most similar works in the literature.} 
In \cite{BRS}, {the authors treat the case of uniformly expanding random
dynamical systems satisfying some uniformity condition and some case of non
uniformly expanding random ones which can be studied by inducing on random
expanding ones.} The general results and the applications are built on
strong assumptions on the deterministic part of the dynamics, while weak
assumptions are considered for the randomness. {\ }%
%In particular the assumption on the dynamics are, in both cases, tailored to guarantee a uniform (with respect to the perturbation e) spectral gap on C� for the transfer operator and invariant densities in C�.
{As a result, the analogous of equations }$(${\ref{finalres1}}$)${\ and }$(${%
\ref{finalres2}}$)${\ give convergence to an element in $C^{1}$ in the
expanding case \cite[Theorem 2.3]{BRS} and roughly $C^{0}$ in the induced
case \cite[Theorem 3.4]{BRS}. Both statements rely on a uniform spectral gap
for the perturbed operators (see \cite[Remark 2.2, Proposition 6.2]{BRS}).} {%
In this work, in some sense we impose strong assumptions on the randomness
for the unperturbed operator but we can work with very weak assumptions on
its deterministic part. Furthermore, note that in Item (LR1) and (LR2) of
Theorem \ref{th:linearresponse} we ask mixing and regularization only on the
unperturbed operator} 
%This allow our approach to be applied to realistic  models of concrete nontrivial physical systems, having a complicated dynamical behavior, as the Belosuv-Zhabotinshy map as we show in Section \ref{sec:finalmain}. 

{Controlling the linear response in a weak sense is instead the point of
view of \cite{HM}. In a nutshell, Hairer and Majda prove linear response for
a Markov process (see \cite[Theorem 2.3]{HM}) in a weak sense showing that
the average of a smooth observable changes smoothly when the process is
perturbed in a suitable way and providing a formula for the derivative. The
functional analytic setting is based on weighted $C^{k}$ norms constructed
by the closure of the space of smooth observables (renormalizing the growth
of the observable), such norms are then used to compare (weakly) the
distance between measures. The general setting is tailored to stochastic
differential equations, and a part of the paper is devoted to the proof that
important examples of this kind satisfy the assumptions of the general
perturbation theorem. This approach also exploits implicitly the
regularizing effect of the noise, like in our paper, but the technical
setting is quite different. We focused on assumptions which can be easily
verified in what we called \emph{discrete time systems with additive noise}.
Our functional setting gives control on the linear response in stronger or
weaker topologies \ ranging from $L^{1}$ to the Wasserstein distance in the
applications we present (see Section \ref{sec:finalmain}) the choice of the
right norm is made according on the kind of perturbation considered (the
right norm to be considered is determined by the norm in which the
"derivative operator converges" as we see in Theorem \ref{th:linearresponse}
and proposition \ref{derivop}).}

%About the differences with respect to \cite{HM}, we draw the attention 
%on the fact that the main focus of the paper is the application to stochastic 
%differential equations, and the assumptions used are adapted to this context. 
%This implies a different choice of the function spaces involved and of some 
%technical solutions. In particular more smoothness is required in the change of  
%the transfer operator during the perturbation (compare Assumption 2 of   
%\cite{HM} with the mixed-norm small perturbation assumptions in 
%\ref{th:linearresponse}), furthermore  we do not require a spectral gap %
%assumption (see Assumption 1  \cite{HM})  but we recover it from the 
%regularization property of the unperturbed operator.}
\end{remark}

In Section \ref{sec:finalmain} We apply the result above in a variety of
examples. We start with a class of topological mixing maps with additive
noise including the mixing piecewise expanding case. In Section \ref{PWappl}
showing how the needed estimates can be achieved and the linear response
result applied to this class.

In section \ref{slat} we consider random rotations. In this case, it is
known (see \cite{Gpre}) that for a rotation of the circle the linear
response does not hold. However once some additive noise is added the system
shows linear response: we produce the required formula.

In section \ref{sec:final}, these results are applied to a model of
Belosuv-Zabotinsky chemical reaction. This is a random dynamical system with
additive noise and whose deterministic part is a map having contracting and
expanding regions. The understanding of the deterministic dynamics of this
map is out of reach with current techniques, on the other hand the system
with noise can be studied rigorously with a computer aided proof (see \cite%
{GMN}) showing that it satisfies the properties required to use the above
statements. In particular we show linear response under suitable
perturbations of the map.

In Section \ref{detappl} we consider an example of a perturbation which is
outside the case of maps with additive noise. We perturb a map with additive
noise by adding the possibility of applying a deterministic map with a small
probability. Here the perturbed operators are not regularizing. We see that
even in this case we get linear response with mild assumptions on the
deterministic part of the dynamics.

Finally, in Section \ref{sec:control} we consider the inverse, control
problem of finding which perturbation of a system results in a wanted
response. We discuss some aspects of this inverse problem, considering it
for the systems with additive noise and allowed perturbations considered in
the paper. We show explicitly solutions to the problem in this case.

\section{Proof of Theorem \protect\ref{th:linearresponse}\label{general}}

In this section we prove the general statement about linear response of
fixed points of transfer operators presented at Theorem \ref%
{th:linearresponse}. As remarked in the introduction, a family of operators {%
might} fail to have linear response, sometime because of lack of
hyperbolicity, sometime because of the non smoothness of the kind of
perturbation which is considered along the family. In particular this is
related to the type of convergence of the derivative operator 
\begin{equation}
\dot{L}f=\lim_{\delta \rightarrow 0}\frac{(L_{\delta }-L_{0})}{\delta }f
\label{ldot1}
\end{equation}%
(see Remark \ref{rrm6}). We remark, as an example, that for deterministic
systems and related transfer operators, if the system is perturbed by moving
its critical values or discontinuities, this will result in a bad
perturbation of the associated transfer operators, and the limit %
\eqref{ldot1} will not converge, unless we consider very coarse topologies. {%
We consider perturbations of the system such that the derivative operator
converges in a weaker or stronger sense i.e. the norm $\Vert \cdot \Vert
_{w} $. Such convergence implies weaker or stronger linear response
statements: we will use this flexibility in the next sections to get
applications in different contexts.}

Let us recall some notations. Mixing can be seen as the convergence to $0$
of iterates of zero average measures, for this we consider the speed of
contraction to zero of iterates of measures in the \ {spaces $V$ and $V_{w}$}
(see Definition \ref{v}). {In the following, with a small abuse of notation,
we will write $\mu \in L^{1}[0,1]$ or $\mu \in BV[0,1]$ to say that $\mu $
has a density $f_{\mu }\in L^{1}[0,1]$ or $f_{\mu }\in BV[0,1]$.} Recall
that we denote by $R(z,L)$ the resolvent operator (see (\ref{d3})). \ Before
the proof of Theorem \ref{th:linearresponse} we need a lemma showing that a
qualitative mixing assumption as in Item $1)$ together with the
regularization supposed at item $2)$ is sufficient to get the exponential
contraction for the space of zero average $L^{1}$ densities.

\begin{definition}
We say that a transfer operator $L$ has exponential contraction of the zero
average space on $L^{1}$\emph{, }if there are $C\geq 0,\lambda <0$ such that 
$\forall g\in V$%
\begin{equation}
\Vert L^{n}(g)\Vert _{1}\leq Ce^{\lambda n}\Vert g\Vert _{1}.  \label{equil}
\end{equation}
\end{definition}

\begin{lemma}
\label{limmin}(mixing and regularization imply exponential contraction) If $%
L:L^{1}([0,1])\rightarrow BV([0,1])$ is continuous and for every $g$ of
bounded variation with $\int_{I}g~dm=0$, 
\begin{equation*}
\lim_{n\rightarrow \infty }\Vert L^{n}g\Vert _{1}=0
\end{equation*}%
then there is $\overline{n}$ such that 
\begin{equation*}
\underset{\|f\|_{1}\leq 1,f\in V}{\sup }\left\Vert L^{\overline{n}%
}f\right\Vert _{1}<1
\end{equation*}%
and then $L$ has exponential contraction of the zero average space on $%
L^{1}. $
\end{lemma}

\begin{proof}[Proof of Lemma \protect\ref{limmin}]
%\notarem{ Suppose $\|L\|_{L^{1}\rightarrow BV}\leq l$.  by this the unit ball 
%$B_{1}$
%of $V$ with the $L^{1}$ norm is sent into \ the set $B_{BV}$ which is the
%radius $l$ ball for the $BV$ norm restricted to $V$ (remark that $V$ is
%preserved by the Markov assumption on $L$).}
{\ By hypothesis, $L:L^{1}\rightarrow BV$ is continuous, thus exists $l\in {%
\mathbb{R}}$ such that $\Vert L\Vert _{L^{1}\rightarrow BV}\leq l$. Since $V$
is preserved by the Markov assumption on $L$, for $f\in V\subset L^{1}$ such
that $\Vert f\Vert _{1}\leq 1$ we have $Lf\in V$ and $\Vert Lf\Vert
_{BV}\leq l$. } By the compact immersion of $BV$ functions in $L^{1}$
(Helly's selection principle) there is a finite $\varepsilon $ net in $L^{1}$
of bounded variation functions $\{g_{i}\,|\,g_{i}\in V\}$ covering $%
B_{BV}:=\{f\in BV\cap V,||f||_{BV}\leq l\}$, i.e. $\forall g\in
B_{BV},\exists \,g_{i}~s.t.~\Vert g-g_{i}\Vert _{1}\leq \varepsilon $.

By the assumptions $\forall i,$ $\lim_{n\rightarrow \infty }\Vert
L^{n}g_{i}\Vert _{1}=0.$ Thus there exists $\overline{n}$ such that $\forall
i,$ $\Vert L^{\overline{n}}g_{i}\Vert _{1}<\frac{1}{10}.$ Consider $g\in
B_{BV}$, since $L$ is a weak contraction in the $L^{1}$ norm$\ \Vert L^{%
\overline{n}}g\Vert _{1}\leq \Vert L^{\overline{n}}g-L^{\overline{n}%
}g_{i}+L^{\overline{n}}g_{i}\Vert _{1}<\varepsilon +\frac{1}{10}$. Because
of the regularization assumption, for every $h$ in $B_{1}:=\{f\in
V,||f||\leq 1\},$ then $\Vert L^{\overline{n}+1}h\Vert _{1}<\frac{1}{10}%
+\varepsilon $, proving the statement. Indeed, for every $n$ we may write $%
n=(\overline{n}+1)m+q$ with $q=0,\ldots \overline{n}-1$. In this case it is
sufficient to note that 
\begin{equation*}
\Vert L^{\overline{n}m+q}h\Vert _{1}\leq \Vert L^{\overline{n}}\Vert
_{L^{1}\rightarrow L^{1}}^{m}\Vert h\Vert _{1}\leq (\frac{1}{10}+\varepsilon
)^{m}\Vert g\Vert _{1}\leq \lbrack (\frac{1}{10}+\epsilon )^{\frac{\lambda }{%
\overline{n}}}]^{m\overline{n}}\Vert g\Vert _{1}\leq Ce^{\tilde{\lambda}%
n}\Vert g\Vert _{1}
\end{equation*}%
for some $C\geq 0,\lambda <0$.
\end{proof}

\begin{corollary}
\label{limmin2}If $L$ is regularizing as above, the exponential contraction
of $V$ in $L^{1}$ implies an exponential contraction on $BV$ too.%
\begin{equation*}
\Vert L^{n}(g)\Vert _{BV[0,1]}\leq C_{2}e^{\lambda n}\Vert g\Vert
_{L^{1}[0,1]}\leq C_{2}e^{\lambda n}\Vert g\Vert _{BV}.
\end{equation*}
\end{corollary}

\begin{proof}
Since $L:L^{1}\rightarrow $ $BV[0,1]$ is continuous $\Vert L^{n}(g)\Vert
_{BV[0,1]}\leq K\Vert L^{n-1}(g)\Vert _{1}\leq KCe^{\lambda (n-1)}\Vert
g\Vert _{L^{1}[0,1]}$ by the exponential contraction in $L^{1}$.
\end{proof}

Now let us consider again the family of operators $L_{\delta }$; let us
recall we suppose that each operator $L_{\delta }$ has a fixed probability
measure in $BV[0,1]$; we are ready to prove the main general statement.

\begin{proof}[Proof of Theorem \protect\ref{th:linearresponse}]
Since $L_{0}$ is regularizing we can apply Lemma \ref{limmin} and deduce
that $L_{0}$ has exponential contraction {of the zero average space} on $%
L^{1}$. Let us first prove that under the assumptions the system has strong
statistical stability in $L^{1}$, that is 
\begin{equation}
\lim_{\delta \rightarrow 0}\|f_{\delta }-f_{0}\|_{1}=0.
\end{equation}

Let us consider for any given $\delta $ a probability measure $f_{\delta }$
such that $L_{\delta }f_{\delta }=f_{\delta }$. Thus%
\begin{eqnarray*}
\Vert f_{\delta }-f_{0}\Vert _{1} &\leq &\|L_{\delta }^{N}f_{\delta
}-L_{0}^{N}f_{0}\|_{1} \\
&\leq &\|L_{\delta }^{N}f_{\delta }-L_{0}^{N}f_{\delta
}\|_{1}+\|L_{0}^{N}f_{\delta }-L_{0}^{N}f_{0}\|_{1}
\end{eqnarray*}%
Since $f_{\delta },f_{0}$ are probability measures, $f_{\delta }-f_{0}\in V$
and $\|f_{\delta }-f_{0}\|_{1}\leq 2$ then we have 
\begin{equation*}
\Vert f_{\delta }-f_{0}\Vert _{1}\leq 2Ce^{\lambda n}+\|L_{\delta
}^{N}f_{\delta }-L_{0}^{N}f_{\delta }\|_{1}.
\end{equation*}%
Next we rewrite the operator sum $L_{0}^{n}-L_{\delta }^{n}$ telescopically 
\begin{equation*}
(L_{0}^{N}-L_{\delta }^{N})=\sum_{k=1}^{N}L_{0}^{N-k}(L_{0}-L_{\delta
})L_{\delta }^{k-1}
\end{equation*}%
so that 
\begin{eqnarray*}
-(L_{\delta }^{N}-L_{0}^{N})f_{\delta }
&=&-\sum_{k=1}^{N}L_{0}^{N-k}(L_{0}-L_{\delta })L_{\delta }^{k-1}f_{\delta }
\\
&=&-\sum_{k=1}^{N}L_{0}^{N-k}(L_{0}-L_{\delta })f_{\delta }.
\end{eqnarray*}%
Let us set 
\begin{equation}
\varepsilon _{\delta }:=\|(L_{0}-L_{\delta })f_{\delta }\|_{1}.  \label{gg}
\end{equation}%
The assumption that $\|f_{\delta }\|_{BV}\leq M,$ together with the small
perturbation assumption (LR3) imply that $\varepsilon _{\delta }\rightarrow
0 $ as $\delta \rightarrow 0.$ {Thus}, recalling that $L_{0}$ has
exponential contraction, 
\begin{equation}
\|f_{\delta }-f_{0}\|_{1}\leq 2Ce^{\lambda N}+\varepsilon _{\delta }N.
\label{eq:strongstability}
\end{equation}

Choosing first $N$ big enough and then $\delta $ small enough we can make $%
\|f_{\delta }-f_{0}\|_{1}$ as small as wanted, proving the stability in $%
L^{1}$.

Let us now consider $R(1,L_{0})$ as an operator$:V_{w}\rightarrow V_{w}$.
Let $f\in V_{w}$ then, by \eqref{d3}, $R(1,L_{0})f=f+\sum_{1}^{\infty
}L_{0}^{i}f$. By the regularization assumption $L_{0}^{i}f\in L^{1}$ for $%
i\geq 1$. Since $L_{0}$ is exponentially contracting on $L^{1}$ and $%
\sum_{1}^{\infty }Ce^{\lambda n}<\infty ,$ the sum $\sum_{1}^{\infty
}L_{0}^{i}f$ \ converges in $V$ with respect to the $L^{1}$ norm, and then
on the weaker norm in $V_{w}\supseteq V$. Thus $R(1,L_{0})={\mathds{1}+}%
\sum_{1}^{\infty }L_{0}^{i}$ is a continuous operator $V_{w}\rightarrow
V_{w} $. Remark that since $\dot{L}f_{0}\in V_{w},$ the resolvent can be
computed at $\dot{L}f_{0}$ . Now we are ready to prove the main statement.
By using that $f_{0}$ and $f_{\delta }$ are fixed points of their respective
operators we obtain that 
\begin{equation*}
({\mathds{1}}-L_{0})\frac{f_{\delta }-f_{0}}{\delta }=\frac{1}{\delta }%
(L_{\delta }-L_{0})f_{\delta }.
\end{equation*}%
By applying the resolvent both sides 
\begin{eqnarray*}
R(1,L_{0})({\mathds{1}}-L_{0})\frac{f_{\delta }-f_{0}}{\delta } &=&R(1,L_{0})%
\frac{L_{\delta }-L_{0}}{\delta }f_{\delta } \\
&=&R(1,L_{0})\frac{L_{\delta }-L_{0}}{\delta }f_{0}+R(1,L_{0})\frac{%
L_{\delta }-L_{0}}{\delta }(f_{\delta }-f_{0})
\end{eqnarray*}%
we obtain that the left hand sides is equal to $\frac{1}{\delta }(f_{\delta
}-f_{0})$. Moreover, with respect to right hand side we observe that,
applying assumption $(LR3)$ eventually, as $\delta \rightarrow 0$%
\begin{equation*}
\left\Vert R(1,L_{0})\frac{L_{\delta }-L_{0}}{\delta }(f_{\delta
}-f_{0})\right\Vert _{w}\leq \Vert R(1,L_{0})\Vert _{V_{w}\rightarrow
V_{w}}K\Vert f_{\delta }-f_{0}\Vert _{1}
\end{equation*}%
which goes to zero thanks to \eqref{eq:strongstability}. Thus considering
the limit $\delta \rightarrow 0$ we are left with 
\begin{equation*}
\lim_{\delta \rightarrow 0}\frac{f_{\delta }-f_{0}}{\delta }=R(1,L_{0})\dot{L%
}f_{0}.
\end{equation*}%
converging in the $\Vert \cdot \Vert _{w}$ norm, which proves our claim.
\end{proof}

\begin{remark}
\label{weakLR0}By inspecting the proof, around \eqref{gg}, we have that the $%
\Vert f_{\delta }\Vert _{BV}\leq M$ assumption can be replaced by $\Vert
f_{\delta }\Vert _{BV}\leq o(\delta ^{-1})$ as $\delta \rightarrow 0$.
\end{remark}

\section{Sistems with additive noise and proof of Proposition \protect\ref%
{derivop} \label{sec:derivative}}

In this section we restrict to random dynamical systems with additive noise.
We consider a random dynamical system on the unit interval {defined by} the
composition of a Borel nonsingular map $T_{0}:[0,1]\rightarrow \lbrack 0,1]$
and a random i.i.d. noise, at each step distributed with a noise kernel $%
\rho _{0}\in BV$ as defined at beginning of Section \ref{adnois}.

We want to apply Theorem \ref{th:linearresponse} to the annealed transfer
operator $L_{0,T_{0}}(f)$ (defined in Section \ref{adnois}) associated to
these systems, for suitable perturbations. In this section the main
objective is to show that the assumption $(LR0)$ and $(LR2)$ of Theorem \ref%
{th:linearresponse} are satisfied by the kind of system we consider, and
study suitable perturbations of the systems, changing the map $T_{0}$ or the
noise kernel $\rho _{0}$ for which the assumption $(LR3)$ Theorem \ref%
{th:linearresponse} is satisfied i.e.: the perturbed operator $L_{\xi
,T_{\delta }}$ is a small perturbation of $L_{0,T_{0}}$ and $\dot{L}f_{0}$
exists in some suitable sense.

Now we show the necessary regularization properties of the boundary
reflecting convolution $\hat{\ast}$ (see Definition \ref{def:hatconv}).
Recall that given two measures $\mu ,\nu \in BS[{\mathbb{R}}]$ their
convolution is defined as (see \cite{Ru}) 
\begin{equation*}
\mu \ast \nu (A)=\int_{{\mathbb{R}}^{2}}{1}_{A}(x+y)d\mu (x)d\nu (y).
\end{equation*}%
which is obviously a measure. Observe that if $d\nu =fd\mu $ with $f\in
L^{1}(\mu )$ the above convolutions reads 
\begin{equation*}
(f\ast \mu )(x)=\int_{{\mathbb{R}}}f(x-y)d\mu (y)
\end{equation*}%
which gives a function and corresponds to an absolutely continuous measure
(one aspect of the regularizing action of the convolution). If both $\mu
,\nu $ are absolutely continuous with respect to Lebesgue i.e. $d\mu
(x)=f(x)dx$ and $d\nu (x)=g(x)dx$ with $f,g\in L^{1}({\mathbb{R}})$ then $%
d(\mu \ast \nu )(x)=(f\ast g)(x)dx$ as expected i.e. the density of the
convolution of the two measures is the convolution of the two densities. The
above properties allow us to mix and match measures in $BS({\mathbb{R}})$
and densities in $L^{1}({\mathbb{R}})$ as needed.

The following Lemma {shows finer} regularization properties of the
convolution $\hat{\ast}$.

\begin{lemma}
\label{convoocopy1} Let $f\in BS$ such that $f([0,1])=0$ and $g\in L^{1}$.
We have 
\begin{equation}
\Vert f\hat{\ast}g\Vert _{W}\leq \Vert f\Vert _{W}\cdot \Vert g\Vert _{1}.
\label{conv1}
\end{equation}%
Furthermore, let $f\in BS,$ s.t. $f([0,1])=0$, $g\in BV$ 
\begin{equation}
\Vert f\hat{\ast}g\Vert _{1}\leq 3\Vert f\Vert _{W}\cdot \Vert g\Vert _{BV}.
\label{conv2}
\end{equation}%
Furthermore, let $f\in L^{1},$ $g\in BV$%
\begin{equation}
\Vert f\hat{\ast}g\Vert _{BV}\leq 9\Vert f\Vert _{1}\cdot \Vert g\Vert _{BV}.
\label{conv3}
\end{equation}
\end{lemma}

\begin{proof}
It is well known that $\pi _{\ast }$ is a weak contraction with respect to
the $L^{1}$ norm \ Let us show that $\pi _{\ast }$ is a weak contraction
with respect to the $\Vert \cdot \Vert _{W}$ norm. If $f\in BS[0,1],$%
\begin{equation*}
\Vert \pi _{\ast }(f)\Vert _{W}={\underset{\Vert g\Vert _{\infty }\leq
1,~Lip(g)\leq 1}{\sup }}\int_{0}^{1}g(x)d(\pi _{\ast }(f)) \\
={\underset{\Vert g\Vert _{\infty }\leq 1,~Lip(g)\leq 1}{\sup }}%
\int_{0}^{1}g(\pi (x))d(f)
\end{equation*}%
since $g\circ \pi $ is $1$-Lipschitz because $\pi $ is a weak contraction
with respect to the Euclidean distance on $\mathbb{R}$, 
\begin{equation*}
{\underset{\Vert g\Vert _{\infty }\leq 1,~Lip(g)\leq 1}{\sup }}%
\int_{0}^{1}g(\pi (x))d(f)\leq \|f\|_{W}.
\end{equation*}

To prove \ref{conv1}, let us consider $f\in L^{1}$ and $g\in L^{1}$. For
these measures we observe that 
\begin{eqnarray*}
\Vert f\hat{\ast}g\Vert _{W} &\leq &\Vert \hat{f}\ast \hat{g}\Vert _{W}=%
\underset{\Vert h\Vert _{\infty }\leq 1,~Lip(h)\leq 1}{\sup }\left\vert
\int_{{\mathbb{R}}}h(t)~(\hat{f}\ast \hat{g})(t)dt\right\vert \\
&=&\underset{\Vert h\Vert _{\infty }\leq 1,~Lip(h)\leq 1}{\sup }\left\vert
\int_{\mathbb{R}}h(t)\int_{\mathbb{R}}\hat{f}(t-\tau )\hat{g}(\tau )d\tau
\,dt\right\vert \\
&\leq &\underset{\Vert h\Vert _{\infty }\leq 1,~Lip(h)\leq 1}{\sup }\int_{%
\mathbb{R}}\left\vert \int_{{\mathbb{R}}}h(t)\hat{f}(t-\tau )dt\right\vert
~\left\vert \hat{g}(\tau )\right\vert d\tau \, \\
&\leq &\int_{\mathbb{R}}\Vert \hat{f}\Vert _{W}~|\hat{g}(\tau )|d\tau \leq
\Vert f\Vert _{W}\cdot \Vert g\Vert _{1}.
\end{eqnarray*}

Now the proof can be extended to the general case where $f\in BS[0,1]$ and $%
g\in L^{1}$ by approximation. First, let us consider $f_{\epsilon }\in L^{1}$
such that $\|f_{\epsilon }-\hat{f} \|_{W}\leq $ $\epsilon $ and $g_{\epsilon
}$ a Lipschitz function such that $\Vert g_{\varepsilon }-\hat{g}\Vert
_{1}\leq \varepsilon $. Then 
\begin{equation*}
\Vert \hat{f}\ast g_{\epsilon }\Vert _{W}\leq \Vert (\hat{f}-f_{\epsilon
})\ast g_{\varepsilon }\Vert _{W}+\Vert f_{\epsilon }\ast g_{\varepsilon
}\Vert _{W}.
\end{equation*}%
Since $g_{\varepsilon }$ is Lipschitz $\Vert (\hat{f}-f_{\epsilon })\ast
g_{\varepsilon }\Vert _{W}$ can be made as small as wanted as $\epsilon
\rightarrow 0$ and the inequality is proved for this case. Now let us remark
that by [\cite{Ru}, Theorem 1.3.2] for each finite Borel measure $f$ and $%
g\in L^{1}$ $\|\hat{f}\ast g\|_{1}\leq \|f\|_{TV} \|g\|_{1}$ where $%
\|f\|_{TV}$ is the total variation of $f$, this shows that in the above
reasoning we can also take the limit for $\varepsilon \rightarrow 0$ since%
\begin{equation*}
\Vert \hat{f}\ast \hat{g}\Vert _{W}\leq \Vert \hat{f}\ast (\hat{g}%
-g_{\varepsilon })\Vert _{W}+\Vert \hat{f}\ast g_{\varepsilon }\Vert _{W}.
\end{equation*}%
and $\Vert \hat{f}\ast \hat{g}-g_{\varepsilon }\Vert _{W}\leq \|\hat{f}%
\|_{TV} \|\hat{g}-g_{\varepsilon }\|_{1}\rightarrow 0.$ Thus by
approximation the statement is proved also for $f\in BS$ and $g\in L^{1}$.

To prove \eqref{conv2} we have to work a little bit harder. Let $g\in BV$,
remark that $\Vert \hat{g}\Vert _{BV}\leq 3\Vert g\Vert _{BV}$ due to the
discontinuity at the borders of $[0,1]$ and to the fact that $\Vert g\Vert
_{\infty }\leq \Vert g\Vert _{BV}$. Consider a $C^{1}$ function $%
g_{\varepsilon }\in V$ with compact support such that $\Vert g_{\varepsilon
}-\hat{g}\Vert _{1}\leq \varepsilon $ and $\Vert g_{\varepsilon }\Vert
_{BV}=\Vert \hat{g}\Vert _{BV}$, then $\Vert f\hat{\ast}g\Vert _{1}\leq
\Vert \hat{f}\ast \hat{g}\Vert _{1}$ and\footnote{%
Again, $\hat{f}\ast \hat{g}$ is in $L^{1}$ and by [\cite{Ru}, Theorem 1.3.2]
for each $h\in L^{1}$ $\|\hat{f}\ast h\|_{1}\leq \|\hat{f}\|\|h\|_{1}$ where 
$\|f\|$ is the total variation of the finite measure $f$ and $%
\|h\|_{1}=\|h\| $ in case of a measure $h\in L^{1}.$}%
\begin{equation*}
\Vert \hat{f}\ast \hat{g}-\hat{f}\ast g_{\varepsilon }\Vert _{1}\leq
\varepsilon
\end{equation*}%
thus we can replace $\hat{g}$ with $g_{\varepsilon }$ up to an error which
is as small as wanted in the estimate we consider.

We now consider the estimate $\Vert \hat{f}\ast g_{\varepsilon }\Vert _{1}$.
Since $g_{\varepsilon }$ is $C^{1}$ with compact support it is bounded and
has bounded derivative (hence Lipschitz constant), then there is $C$ such
that for every $f_{\epsilon }\in L^{1}$ supported in $[0,1]$ such that $%
\|f_{\epsilon }-\hat{f}\|_{W}\leq $ $C\epsilon $ it holds $|\hat{f}\ast
g_{\varepsilon }(x)-f_{\epsilon }\ast g_{\varepsilon }(x)|\leq \epsilon $
for every $x\in {\mathbb{R}}$, thus 
\begin{equation*}
\Vert f_{\epsilon }\ast g_{\epsilon }-\hat{f}\ast g_{\varepsilon }\Vert
_{1}\leq \varepsilon
\end{equation*}%
and we can also replace $\hat{f}$ with $f_{\epsilon }$ in our main estimate.

Now, recalling $C^{1}$ functions are the integral of their derivative, 
\begin{equation*}
\begin{split}
\Vert f_{\epsilon }\ast g_{\varepsilon }\Vert _{1}& =\int_{\mathbb{R}%
}\left\vert (f_{\epsilon }\ast g_{\varepsilon })(t)\right\vert dt={\underset{%
\Vert h\Vert _{\infty }\leq 1}{\sup |}}\int_{{\mathbb{R}}}h(t)(f_{\epsilon
}\ast g_{\varepsilon })(t)dt| \\
& \leq {\underset{\Vert h\Vert _{\infty }\leq 1}{\sup }|}\int_{{\mathbb{R}}%
}\int_{0}^{t}h(l)dl~(f_{\epsilon }\ast g_{\varepsilon })^{\prime }(t)dt| \\
& ={\underset{\Vert h\Vert _{\infty }\leq 1}{\sup }|}\int_{{\mathbb{R}}%
}\int_{0}^{t}h(l)dl(f_{\epsilon }\ast g_{\varepsilon }^{\prime })(t)dt|.
\end{split}%
\end{equation*}

Note that the last integrand above, if we let $\int_{0}^{t}h(l)=s(t)$ (note
that $s(t)$ is a $1-$Lipschitz function) can be rewritten as%
\begin{equation*}
|\int_{{\mathbb{R}}}s(t)\left( \int_{\mathbb{R}}g_{\varepsilon }^{\prime
}(r)f_{\epsilon }(t-r)dr\right) dt|\leq \int_{{\mathbb{R}}}\int_{\mathbb{R}%
}|s(t)f_{\epsilon }(t-r)|dt\left\vert g_{\varepsilon }^{\prime
}(r)\right\vert dr.
\end{equation*}

Since $f_{\epsilon }$ is a zero average density with support of diameter $%
\leq 1,$ for every translate $f_{\epsilon }(t-r)$ 
\begin{equation*}
\int_{\mathbb{R}}|s(t)f_{\epsilon }(t-r)|dt\leq \|f_{\epsilon }\|_{W}
\end{equation*}%
(on the support of $f_{\epsilon }(t-r),\ s(t)=C+\hat{s}(t)$ with $\|\hat{s}%
\|_{\infty }\leq 1$ and $\int_{\mathbb{R}}|Cf_{\epsilon }(t-r)|dt=0$) then 
\begin{equation*}
\Vert f_{\epsilon }\ast g_{\varepsilon }\Vert _{1}\leq \Vert f_{\epsilon
}\Vert _{W}\Vert g_{\varepsilon }^{\prime }\|_{1}\leq \Vert f\Vert _{W}\Vert
g_{\varepsilon }\Vert _{BV}.
\end{equation*}%
Hence $\Vert f\hat{\ast}g\Vert _{1}\leq \Vert \hat{f}\ast \hat{g}\Vert
_{1}\leq 3\Vert f\Vert _{W}\Vert g\Vert _{BV}+3\varepsilon $ for each $%
\varepsilon $ which gives the statement.

About Equation $($\ref{conv3}$)$ we remark that $\Vert f\ast g\Vert
_{BV}\leq \Vert f\Vert _{1}\cdot \Vert g\Vert _{BV}$ \ is well known. For
the convolution on the interval let us remark that $\Vert \pi ^{\ast }(\mu
)\Vert _{BV}\leq 3\Vert \mu \Vert _{BV}$ and $\Vert \hat{g}\Vert _{BV}\leq
3\Vert g\Vert _{BV}.$
\end{proof}

Lemma \ref{convoocopy1} provides sufficient estimates to prove that the
transfer operator associated to a system with additive noise is regularizing
in the sense required by the assumption $(LR2)$ of Theorem \ref%
{th:linearresponse} under suitable mild assumptions on $T$, as we show in
the following.

\begin{corollary}
\label{regcor}Let $L_{0,T}$ be the trasfer operator of a system with
additive noise whose deterministic part is $T:[0,1]\rightarrow \lbrack 0,1]$
and the noise is distributed according to a kernel $\rho _{0}\in BV[0,1]$. \
The following holds.

\begin{description}
\item[1] If $T$ is nonsingular then $L_{0,T}$ is regularizing from $L^{1}$
to Bounded Variation i.e. $L_{0,T}{:L^{1}\rightarrow BV}$ is continuous.

\item[2] If $T$ is Lipschitz, then $L_{0,T}$ is regularizing from ${BS}$
endowed with the Wasserstein norm to $\ L^{1}$ i.e. $L_{0,T}{:BS\rightarrow
L^{1}}$ is continuous.
\end{description}
\end{corollary}

\begin{proof}
The item 1 directly follows from the definitions. Let us take $f$ with $%
||f||_{1}\leq 1$, then $L_{0,T}(f):=\rho _{0}\hat{\ast}L_{T}(f)$ then \ $%
||L_{T}(f)||_{1}\leq 1$ and by $($\ref{conv3}$)$ $||\rho _{0}\hat{\ast}%
L_{T}(f)||_{BV}\leq 9||\rho _{0}||_{BV}$.

Similarly for the item 2 we remark that if $K$ is the Lipschitz constant of $%
T$ then $||L_{T}(f)||_{W}\leq \max (K,1)||f||_{W},$ by $($\ref{conv2}$)$ we
get $||\rho _{0}\hat{\ast}L_{T}(f)||_{1}\leq \max (K,1)||f||_{W}||\rho
_{0}||_{BV}$.
\end{proof}

Another consequence of the above regularization properties is the existence
of a fixed point for $L_{\xi ,T_{\delta }}$ in $BV$ whose variation is
bounded by the variation of the noise kernel. This will imply that
assumption $(LR0)$ is satisfied by families of systems whose kernel has
uniformly bounded variation.

\begin{lemma}
\label{lem:fixedpoint} Let us suppose $T_{\delta }$ being a nonsingular map,
let $\rho _{\xi }\in BV$ and let {%
\begin{equation}
L_{\xi ,T_{\delta }}(f):=\rho _{\xi }\hat{\ast}L_{T_{\delta }}(f)
\end{equation}%
be a transfer operator of a system with additive noise as in }$(${\ref%
{eq:transop}}$)${. Then }there exists $f_{\xi ,\delta }\in BV$ such that $%
L_{\xi ,T_{\delta }}f_{\xi ,\delta }=f_{\xi ,\delta }$ and%
\begin{equation*}
||f_{\xi ,\delta }||_{BV}\leq 9||\rho _{\xi }||_{BV}.
\end{equation*}
\end{lemma}

\begin{proof}
The proof is similar to the proof of Lemma \ref{extuniq}. Let us denote $%
L:=L_{\xi ,T_{\delta }}$ for short. Let us consider the iterates $L^{n}(m)$
of the Lebesgue measure $m$. Because of $($\ref{conv3}$)$ these iterates
have a Bounded Variation density, indeed since $L$ is a Markov operator, $%
\Vert L^{n}(m)\Vert _{1}\leq 1$ for all $n$, and $L:L^{1}\rightarrow BV[0,1]$
is continuous we have that 
\begin{eqnarray*}
\Vert L^{n}(m)\Vert _{BV} &\leq &\Vert L\Vert _{L^{1}\rightarrow BV}\Vert
L^{n-1}(m)\Vert _{1}\leq \Vert L\Vert _{L^{1}\rightarrow BV} \\
&\leq &9||\rho _{\xi }||_{BV}
\end{eqnarray*}%
for every $n$. Now the proof continues exactly like in Lemma \ref{extuniq}.
We obtain a limit measure for the Cesaro averages $f_{\xi ,\delta }\in BV$
with $||f_{\xi ,\delta }||_{BV}\leq 9||\rho _{\xi }||_{BV}.$
\end{proof}

\subsection{Perturbing the map}

\label{secptm}

We now consider the case in which \ the system is perturbed by changing $T$
while keeping the noise unchanged. Hence we change \textquotedblleft
deterministically\textquotedblright\ the deterministic part of the system.
Let $T_{\delta }:[0,1]\rightarrow \lbrack 0,1]$ be a family of
\textquotedblleft small\textquotedblright\ perturbations of $T_{0}$ (in a
sense which will be made precise below) parametrized by $\delta \in \lbrack
0,\overline{\delta }]$. Consider the family of transfer operators $%
L_{T_{\delta }}:L^{1}([0,1])\rightarrow L^{1}([0,1])$, $\delta \in \left[ 0,%
\overline{\delta }\right] $ and \ the associated operators with noise $%
L_{\xi ,T\delta }$ defined as before as 
\begin{equation*}
L_{\xi ,T_{\delta }}(f):=\rho _{\xi }\hat{\ast}L_{T_{\delta }}(f).
\end{equation*}

We now compute the structure of the derivative operator and prove that these
perturbations, under suitable assumptions, are small, satisfying assumption $%
(LR3)$ of Theorem \ref{th:linearresponse}. To this purpose we have the
following

\begin{proposition}
\label{nearness} Let $f\in L^{1}$, let $D_{\delta }:[0,1]\rightarrow \lbrack
0,1]$ be a family of continuous bijections, such that 
\begin{equation*}
D_{\delta }={\mathds{1}}+\delta S
\end{equation*}%
where $S$ is a $1$-Lipschitz map such that $S(0)=S(1)=0$ with support $\chi
_{S},$ finite union of intervals. Let $T:[0,1]\rightarrow \lbrack 0,1]$ be a
nonsingular Borel map and $L_{T}$ its transfer operator. Let $L_{D_{\delta
}\circ T}$ be the transfer operator of the map $D_{\delta }\circ T$. Then 
%Suppose $%
%L_{T}:BV(\chi _{S})\rightarrow BV([0,1])$ is continuous\footnote{%
%There is $K\geq 0$ such that $var(1_{\chi _{S}}\cdot L_{T}f)\leq Kvar(f)$
%for each $f\in BV[0,1]$.}. Let $L_{D_{\delta }\circ T}$ be the transfer
%operator of the map $D_{\delta }\circ T$ then there is $K\geq 0$ such that 
\begin{equation}
\Vert L_{D_{\delta }\circ T}f-L_{T}(f)\Vert _{W}\leq \delta \Vert f\Vert
_{1}.  \label{nearness1}
\end{equation}%
%
%
%
%
%
%
%
%
%
%
%
%
%
%
%
%
%
%
%
%
%
%
%
%
%\begin{equation}
%\|L_{D_{\delta }\circ T}f-L_{T}(f)\|_{1}\leq \delta \|f\|_{BV}.
%\label{nearness2}
%\end{equation}
\end{proposition}

\begin{proof}
Equation $($\ref{nearness1}$)$ follows by the remark that for each $g\in
L^{1}$, $\Vert L_{D_{\delta }}g-g\Vert _{W}\leq \delta \Vert g\Vert _{W}\leq
\delta \Vert g\Vert _{1},$then posing $g=L_{T}(f),$ $\ \Vert L_{D_{\delta
}\circ T}f-L_{T}(f)\Vert _{W}\leq \delta \Vert L_{T}(f)\Vert _{1}\leq \delta
\Vert f\Vert _{1}.$

%To prove Equation \ref{nearness2} let us remark that for $\delta $ small
%enough 
%\begin{equation*}
%\int |L_{D_{\delta }\circ T}f-L_{T}(f)|dm\leq \int |L_{D_{\delta }\circ
%T}(f\ast 1_{\chi _{S}})-L_{T}(f\ast 1_{\chi _{S}})|dm+\int |L_{D_{\delta
%}\circ T}(f\ast 1_{\chi _{S}^{c}})-L_{T}(f\ast 1_{\chi _{S}^{c}})|dm
%\end{equation*}%
%and $\int |L_{D_{\delta }\circ T}(f\ast 1_{\chi _{S}^{c}})-L_{T}(f\ast
%1_{\chi _{S}^{c}})|dm=0.$ On the other hand $\int |L_{D_{\delta }\circ
%T}(f\ast 1_{\chi _{S}})-L_{T}(f\ast 1_{\chi _{S}})|dm=\int |L_{D_{\delta
%}}(L_{T}(f\ast 1_{\chi _{S}}))-L_{T}(f\ast 1_{\chi _{S}})|dm$ since $%
%L_{T}:BV(\chi _{S})\rightarrow BV([0,1])$ is continuous then $\overline{f}%
%:=L_{T}(f\ast 1_{\chi _{S}})\in BV[0,1]$ the statement follow by the fact
%that $\|L_{D_{\delta }}(\overline{f})-\overline{f}\|_{1}\leq K\delta $.%
%\footnote{\ref{nearness2} is not used anymore in the paper. maybe is better
%to remove it?}
\end{proof}

\begin{proposition}
\label{lem_wderivative} Let $D_{\delta }:[0,1]\rightarrow \lbrack 0,1]$ be a
family of continuous bijections as in Proposition \ref{nearness}. Let $%
L_{D_{\delta }}$ be the associated transfer operator and suppose $f\in
BV(\chi _{S}),$ then 
\begin{equation*}
\lim_{\delta \rightarrow 0}\left\Vert \frac{L_{D_{\delta }}f-f}{\delta }%
-(-f\cdot S)^{\prime }\right\Vert _{W}=0
\end{equation*}%
where $(f\cdot S)^{\prime }$ is meant in the weak sense (it is a measure,
the derivative of a $BV$ function).
\end{proposition}

\begin{proof}
Let us remark that since $S(0)=S(1)=0$ then $\frac{L_{D_{\delta }}f-f}{%
\delta }-(-f\cdot S)^{\prime }$ represents a zero average measure while $%
\frac{L_{D_{\delta }}f-f}{\delta }\in L^{1}$. \ Thus 
\begin{equation*}
\Vert \frac{L_{D_{\delta }}f-f}{\delta }-(-f\cdot S)^{\prime }\Vert _{W}\leq 
\underset{g\in C^{1},Lip(g)\leq 1}{\sup }\int_{0}^{1}g~d[\frac{L_{D_{\delta
}}f-f}{\delta }-(-f\cdot S)^{\prime }].
\end{equation*}%
For every $g$ such that $g(0)=0$ and $Lip(g)\leq 1$, it holds%
\begin{eqnarray*}
\int gd\left[ \frac{L_{D_{\delta }}f-f}{\delta }-(-fS)^{\prime }\right]
&=&\int gd[(fS)^{\prime }]+\int \frac{gL_{D_{\delta }}f-gf}{\delta }dm \\
&=&\int gd[(fS)^{\prime }]+\int \frac{g\circ D_{\delta }-g}{\delta }dfm
\end{eqnarray*}%
where we used the duality between the transfer operator and the composition
operator. Let us observe that the rightmost element can be written as 
\begin{equation*}
\int \frac{g\circ D_{\delta }-g}{\delta }dfm=\int \frac{{g(x+\delta
S(x))-g(x)}}{\delta }dfm;
\end{equation*}%
now, pointwise it holds $\frac{{g(x+\delta S(x))-g(x)}}{\delta }\underset{%
\delta \rightarrow 0}{\rightarrow }g^{\prime }(x)S(x)$ and the convergence
is dominated by $\frac{|{g(x)-g(x+\delta S(x))|}}{{\delta }}{\leq }\underset{%
x}{{\sup }}{(|S(x)|)<\infty }$. By the Lebesgue convergence theorem hence $%
\int \frac{{g(x)-g(x+\delta S(x))}}{\delta }dfm\rightarrow \int g^{\prime
}(x)S(x)~dfm$. \ Since $g(0)=0$ and $S(1)=0,$ we obtain, using integrations
by parts again, that 
\begin{equation}
\int gd[(fS)^{\prime }]+\int g^{\prime }(x)S(x)dfm=0
\end{equation}%
proving the statement.
\end{proof}

\begin{corollary}
\label{cor}Let $T:[0,1]\rightarrow \lbrack 0,1]$ be a nonsingular Borel map
and $L_{T}$ its transfer operator. Suppose $L_{T}:BV([0,1])\rightarrow
BV(\chi _{S})$ is continuous\footnote{%
There is $K\geq 0$ such that $\|1_{\chi _{S}}\cdot L_{T}f\|_{BV}\leq
K\|f\|_{BV}$ for each $f\in BV[0,1]$. This condition is sufficient to ensure
that $L_{T}(f)S$ is a $BV$ function and then $(L_{T}(f)S)^{\prime }\in BS.$}%
. Let $L_{D_{\delta }\circ T}$ as before then 
\begin{equation*}
\lim_{\delta \rightarrow 0}\|\frac{L_{D_{\delta }\circ T}f-L_{T}f}{\delta }%
-(-L_{T}(f)S)^{\prime }\|_{W}=0.
\end{equation*}
\end{corollary}

\begin{proof}
Since $L_{T}f$ is of bounded variation on $\chi _{S}$, it is sufficient to
apply Proposition \ref{lem_wderivative} to $L_{T}f,$ noticing that $%
L_{D_{\delta }\circ T}=L_{D_{\delta }}(L_{T}f)$.
\end{proof}

\begin{remark}
In the case where $L_{T}f\notin BV(\chi _{S})$ or $L_{T}(f)\cdot S$ is not
of Bounded Variation, $( L_{T}(f_0) \cdot S)^{\prime }$ can be interpreted
as a distribution in the dual of the Lipschitz functions.
\end{remark}

Applying the convolution equation \ref{conv2} to Corollary \ref{cor} and
Proposition \ref{nearness} we get directly the convergence of the derivative
and small perturbation of the operator as requested at Item $3)$ of Theorem %
\ref{th:linearresponse}.

\begin{proposition}
\label{pertmap}Let $D_{\delta }={\mathds{1}}+\delta S$ be a bilipschitz
homeomorphism near the identity, $[0,1]\rightarrow \lbrack 0,1]$, where $S$
is a $1$-Lipschitz map such that $S(0)=S(1)=0$. Denote with $\chi _{S}$ the
support of $S,$ suppose $\chi _{S}$ is a finite union of intervals. Suppose $%
L_{T}:BV[0,1]\rightarrow BV(\chi _{S})$ is continuous. Let $T_{\delta
}=D_{\delta }\circ T$ so that 
\begin{equation*}
L_{\xi ,T_{\delta }}f:=\rho _{\xi }\hat{\ast}L_{D_{\delta }\circ T}f
\end{equation*}%
then $\|L_{\xi ,D_{\delta }\circ T_{0}}f-L_{\xi ,T_{0}}(f)\|_{1}\leq \delta
\|f\|_{1}$ and the following limit defining the derivative operator of such
a system converges 
\begin{equation}
\underset{\delta \rightarrow 0}{\lim }\left\Vert \frac{(L_{\xi ,T_{\delta
}}-L_{\xi ,T_{0}})}{\delta }f_{0}-\rho _{\xi }\hat{\ast}(-L_{T}(f_{0})S)^{%
\prime }\right\Vert _{1}=0.  \label{eq:noiseresponse}
\end{equation}
\end{proposition}

In this formula $(L_{T}(f_{0})S)^{\prime }$ should be interpreted as a
measure.

\begin{remark}
\label{bvcase}In the case where $L_{T}:BV[0,1]\rightarrow BV[0,1]$ is
continuous (like for example in piecewise expanding maps, but this is not
restricted to that, a bounded amount of contraction could be allowed) then $%
L_{T}:BV([0,1])\rightarrow BV(\chi _{S})$ is automatically continuous.
\end{remark}

\begin{remark}
In the case where $L_{T}(f_{0})S$ is not of bounded variation and $%
(L_{T}(f_{0})S)^{\prime }$ only converges in the dual of the space of
Lipschitz densities. In this case $\rho _{\xi }\hat{\ast}(-L_{T}(f_{0})S)^{%
\prime }$ could not be defined everywhere when $\rho _{\xi }$ is only a
bounded variation kernel. However, in the case where $\rho _{\xi }$ is
Lipschitz we have that $\rho _{\xi }\hat{\ast}(-L_{T}(f_{0})S)^{\prime }$ is
well defined and the limit $($\ref{eq:noiseresponse}$)$ holds with a
convergence in $L^{\infty }$.
\end{remark}

\begin{remark}
\label{rem_GP17} We remark that the theorem above requires very mild
assumptions. To be more explicit compare it with the following result of 
\cite{GP}.

\begin{proposition}[\protect\cite{GP}]
Let $T_{\delta }$ be a family of nearby expanding maps such that 
\begin{equation*}
T_{\delta }=T_{0}+\delta p+o_{C^{3}}(\delta )
\end{equation*}%
where $p\in C^{3}$ and $o_{C^{3}}(\delta )$ is the error term with respect
to the $C^{3}$ norm of $\frac{1}{\delta }(T_{\delta }-T_{0})-\varepsilon $,
and for their respective operators it holds a linear response statement. Let 
$\bar{L}_{\delta }$ be their transfer operators and $\hat{L}$ the related
derivative operator. Let $w\in C^{3}(X,\mathbb{R})$. For each $x\in X$ we
can write 
\begin{equation*}
\begin{split}
\hat{L}w(x)& =\lim_{\delta \rightarrow 0}\left( \frac{\bar{L}_{\delta }w(x)-%
\bar{L}_{0}w(x)}{\delta }\right) \\
& =-\bar{L}_{0}\left( \frac{w\varepsilon ^{\prime }}{T_{0}^{\prime }}\right)
(x)-\bar{L}_{0}\left( \frac{\varepsilon w^{\prime }}{T_{0}^{\prime }}\right)
(x)+\bar{L}_{0}\left( \frac{\varepsilon T_{0}^{\prime \prime }}{%
T_{0}^{\prime 2}}w\right) (x)
\end{split}%
\end{equation*}%
and the convergence is also in the $C^{1}$ topology.
\end{proposition}

In this case by using this very explicit description of the linear response
the equation \eqref{eq:noiseresponse} becomes 
\begin{equation*}
\underset{\delta \rightarrow 0}{\lim }\left\Vert \frac{(L_{\xi ,T_{\delta
}}-L_{\xi ,T_{0}})}{\delta }f_{0}-\rho _{\xi }\ast (-L_{T}(f_{0})S)^{\prime
}\right\Vert _{1}=0.
\end{equation*}

In fact it is easy to see that the result above can be included in the
present framework by considering $\rho_\xi * \bar{L}_{\delta } $ and
obtaining, as a result, $\rho_\xi * \hat{L} $.
\end{remark}

% \begin{remark}
% Is it possible to adapt the previous theorem, and in particular the operator 
% $N_{\xi ,\delta }$, to satisfy the needs of computational methods e.g. in
% the sense that round off errors can be thought as noise. In particular by
% correctly choosing the parameters... one can set... 
% \begin{equation*}
% N_{\xi ,\delta }=\Pi _{\delta }N_{\xi }\Pi _{\delta }\mathcal{L}_{\xi }\Pi
% \end{equation*}
% \end{remark}

Applying Theorem \ref{th:linearresponse} we have hence the following
corollary, summarizing the existence of linear response for systems with
additive noise and perturbations as considered in Proposition \ref{pertmap}.

\begin{corollary}
\label{cormap2}Let $L_{0,T_{0}}$ be the transfer operator of a system with
additive noise, as above. Suppose that $L_{0,T_{0}}$satisfies assumption
LR1. Let $T_{\delta }=D_{\delta }\circ T$ be a perturbation of $T_{0}$ as in
Proposition \ref{pertmap}. Let $L_{0,T_{\delta }}$ be the transfer operator
associated to the system with additive noise whose deterministic part is $%
T_{\delta }$ \ and the noise is distributed according to $\rho _{0}\in BV.$
Let $f_{\delta }$ be fixed probability measures of the transfer operators $%
L_{0,T_{\delta }},$ i.e. $L_{0,T_{\delta }}f_{\delta }=f_{\delta }$ then $%
R(z,L_{0,T_{0}}):V\rightarrow $ $V$ is a continuous operator and we have the
following Linear Response formula 
\begin{equation}
\lim_{\delta \rightarrow 0}\left\Vert \frac{f_{\delta }-f_{0}}{\delta }%
-R(1,L_{0,T_{0}})\rho _{0}\hat{\ast}(-L_{T_{0}}(f)S)^{\prime }\right\Vert
_{1}=0.  \label{finalres1.2}
\end{equation}
\end{corollary}

\subsection{Perturbing the noise}

\label{secptn}

In this section we consider the situation where {given} a system with noise,
as described above, one is interested in understanding how the stationary
measure changes if the structure of the noise changes. In particular to
study what happens if the radius of the noise changes. Consider hence a
nonsingular map $T$ and a family of kernels {$\rho _{\xi }\in BV$, $\xi \in
\lbrack 0,\bar{\xi})$ } and the associated family of transfer operators $%
L_{\xi ,T}f:=\rho _{\xi }\hat{\ast}L_{T}f$ \ where the additive noise is
changed as $\xi $ changes. {Note that we do not mean to consider
\textquotedblleft zero-noise"-limits, as the transfer operator without noise
is not typically a regularizing one (remark that from the technical point of
view, the zero noise limit correspond to }$\rho _{0}$ being equal to the
dirac delta placed in \thinspace $0$, which is not a $BV$ density){. }Let us
suppose that $f_{\xi }$ is the stationary probability measure for $L_{\xi
,T}.$ Here the {important} limit {to be considered} {is} the
\textquotedblleft derivative kernel" $\dot{\rho}:=\underset{\xi \rightarrow 0%
}{\lim }\frac{\rho _{\xi }-\rho _{0}}{\xi }.$ {\ We will see interesting
cases where the limit converges in the Wasserstein $\Vert \cdot \Vert _{W}$
norm (see e.g. Example \ref{exnoise}). In this case applying Theorem \ref%
{th:linearresponse} we get a linear response statement with convergence in
the $\Vert \cdot \Vert _{W}$ norm (see Corollary \ref{cornoise} and
Corollary \ref{cor:tentspertubartion}).} In order to apply the theorem, we
show how its assumption $(LR3)$ can be verified in this case.

First we show that under simple and natural assumptions on the perturbation
of the noise kernel we have that the perturbed transfer operator is a small
perturbation of the unperturbed one (first part of $(LR3)$).

\begin{proposition}
\label{nearness copy(1)} Let $f\in L^{1}$, suppose $\|\rho _{\xi }-\rho
_{0}\|_{1}\leq C\xi $, then 
\begin{equation}
\|L_{\xi ,T}f-L_{0,T}f\|_{W}\leq \|L_{\xi ,T}f-L_{0,T}f\|_{1}\leq C\xi
\|f\|_{1}.
\end{equation}
\end{proposition}

\begin{proof}
The proof is a direct application of the well known fact $\Vert \rho \ast
g\Vert _{1}\leq \Vert \rho \Vert _{1}\Vert g\Vert _{1}$. By this it also
holds 
\begin{equation}
\Vert \rho \hat{\ast}g\Vert _{1}=\pi _{\ast }(\hat{\rho}\ast \hat{g})\leq
\Vert \rho \Vert _{1}\Vert g\Vert _{1}  \label{convl1}
\end{equation}

and $\Vert \rho _{0}\hat{\ast}g-\rho _{\xi }\hat{\ast}g\Vert _{1}\leq \Vert
\rho _{\xi }-\rho _{0}\Vert _{1}\Vert g\Vert _{1}$.
\end{proof}

Then we consider the convergence of the derivative operator (second part of $%
(LR3)$).

\begin{proposition}
\label{plusnoise}Suppose $||~||_{w}$ is either the $L^{1}$ or the $||~||_{W}$
norm. Suppose there is a Borel measure with sign $\dot{\rho}$ such that 
\begin{equation*}
\underset{\xi \rightarrow 0}{\lim }\left\Vert \frac{\rho _{\xi }-\rho _{0}}{%
\xi }-\dot{\rho}\right\Vert _{w}=0.
\end{equation*}%
Suppose $T$ is nonsingular and $f_{0}\in L^{1}[0,1]$, then 
\begin{equation*}
\underset{\xi \rightarrow 0}{\lim }\left\Vert \frac{(L_{\xi ,T}-L_{0,T})}{%
\xi }f_{0}-\dot{\rho}\hat{\ast}L_{T}(f_{0})\right\Vert _{w}=0.
\end{equation*}
\end{proposition}

\begin{proof}
The proof is a direct computation%
\begin{equation*}
\frac{\rho _{\xi }\hat{\ast}L_{T}f-\rho _{0}\hat{\ast}L_{T}f}{\xi }=\frac{%
\rho _{\xi }-\rho _{0}}{\xi }\hat{\ast}L_{T}f.
\end{equation*}%
We have that $L_{T}f_{0}\in L^{1}$. In the case where $||~||_{w}=||~||_{W}$,
by $($\ref{conv1}$)$ we get the statement. In the case $||~||_{w}$ is the $%
L^{1}$ norm we have the same result by using $(\ref{convl1})$.
\end{proof}

We remark that by this computation, assumption $(LR3)$ of Theorem \ref%
{th:linearresponse} is satisfied with a derivative operator $\dot{L}(f)=\dot{%
\rho}\hat{\ast}L_{T}(f).$

By Theorem \ref{th:linearresponse} and Propositions \ref{plusnoise}, \ref%
{nearness copy(1)} in the case the choice of the weak norm is $%
||~||_{w}=||~||_{W}$, we get the following linear response statement for
perturbations of the noise distribution $\rho _{0}.$

\begin{corollary}
\label{cornoise} Let $L_{0,T_{0}}$ be the transfer operator of a system with
additive noise, as above. Suppose that $L_{0,T_{0}}$satisfies assumption LR1
(mixing).\ Suppose $T_{0}$ is Lipschitz, $\Vert \rho _{\xi }-\rho _{0}\Vert
_{1}\leq C\xi $, and that there exists $\dot{\rho}\in BS$ such that 
\begin{equation*}
\underset{\xi \rightarrow 0}{\lim }\left\Vert \frac{\rho _{\xi }-\rho _{0}}{%
\xi }-\dot{\rho}\right\Vert _{W}=0
\end{equation*}%
then $R(z,L_{0,T_{0}}):V_{W}\rightarrow $ $V_{W}$ is a continuous operator
and we have the following Linear Response formula 
\begin{equation}
\lim_{\delta \rightarrow 0}\left\Vert \frac{f_{\delta }-f_{0}}{\delta }%
-R(1,L_{0})\dot{\rho}\hat{\ast}L_{T}(f)\right\Vert _{W}=0.
\end{equation}
\end{corollary}

If the derivative operator converges in the $L^{1}$ norm the choice of the
weak norm is $||~||_{w}=||~||_{1}$ and we get

\begin{corollary}
\label{cornoise copy(1)}Let $L_{0,T_{0}}$ be the transfer operator of a
system with additive noise, as above. Suppose that $L_{0,T_{0}}$satisfies
assumption $(LR1)$ (mixing).\ Suppose $T_{0}$ is nonsingular, $\Vert \rho
_{\xi }-\rho _{0}\Vert _{1}\leq C\xi $, and that there exists $\dot{\rho}\in
L^{1}$ such that 
\begin{equation*}
\underset{\xi \rightarrow 0}{\lim }\left\Vert \frac{\rho _{\xi }-\rho _{0}}{%
\xi }-\dot{\rho}\right\Vert _{1}=0
\end{equation*}%
then $R(z,L_{0,T_{0}}):V\rightarrow $ $V$ is a continuous operator and we
have the following Linear Response formula 
\begin{equation}
\lim_{\delta \rightarrow 0}\left\Vert \frac{f_{\delta }-f_{0}}{\delta }%
-R(1,L_{0})\dot{\rho}\hat{\ast}L_{T}(f)\right\Vert _{1}=0.
\end{equation}
\end{corollary}

\begin{example}
\label{exnoise}(perturbing the uniform noise) Let us now consider a concrete
and natural example of perturbation of the noise kernel in a system with
additive noise: the case where we consider uniform noise in a certain small
radius and change the radius. Let us hence consider a system with noise
kernel $\rho _{0}=a^{-1}1_{[-a/2,a/2]}$ with $a>0$ and slightly perturb it
to $\rho _{\xi }=(a+\xi )^{-1}1_{[-(a+\xi )/2,(a+\xi )/2]}$ \ By Proposition %
\ref{plusnoise} we consider $\dot{\rho}=\underset{\xi \rightarrow 0}{\lim }%
\frac{\rho _{\xi }-\rho _{0}}{\xi }$. We have convergence of this limit in
the $\Vert \cdot \Vert _{W}$ norm. A little computation yields 
\begin{equation}
\dot{\rho}=-a^{-2}1_{[-a/2,a/2]}+\frac{a^{-1}}{2}\delta _{-a/2}+\frac{a^{-1}%
}{2}\delta _{a/2}  \label{rodot}
\end{equation}%
where $\delta _{x}$ is the Dirac delta measure placed at $x$.
\end{example}

\section{Examples of application}

\label{sec:finalmain}

To show the flexibility and applicability of our theory we produce a family
of examples, namely: a family of maps satifying a kind of topological mixing
condition, circle rotations, a nontrivial map coming from real-life
applications which has very nonuniform properties, an iterated functions
system with deterministic and random components.

\subsection{Maps which are eventually onto and additive noise}

\label{PWappl}

Let us consider a family of nonsingular maps $T_{\delta }$ with additive
noise. We will suppose that $T_{0}$ is a nonsingular piecewise $C^{2}$ map%
\footnote{%
The space $[0,1]$ can be decomposed into a union of intervals $I_{i}$ such
that in every set $\overline{I_{i}}$ the map $T_{0}$ can be extended to a $%
C^{2}$ function $\overline{I_{i}}\rightarrow \lbrack 0,1].$} having good
distortion properties (which will be specified below) and $T_{\delta }$ is
obtained by composition with a family of diffeomorphism constructed as in
Section \ref{secptm}. More precisely let us suppose:

\begin{itemize}
\item[A1] $T_{0}$ is a nonsingular piecewise $C^{2}$ map whose transfer
operator $L_{T_{0}}:BV[0,1]\rightarrow BV[0,1]$ is continuous: there is $%
C\geq 0$ such that%
\begin{equation*}
||L_{T_{0}}f||_{BV}\leq C||f||_{BV}.
\end{equation*}

\item[A2] $T_{0}$ is eventually onto: for each open interval $I\subseteq
\lbrack 0,1]$ there is $n$ such that $T_{0}^{n}(I)=[0,1]$.

\item[A3] $T_{\delta }=D_{\delta }\circ T_{0}$ where $D_{\delta
}:[0,1]\rightarrow \lbrack 0,1]$ be a family of continuous bijections, such
that 
\begin{equation*}
D_{\delta }={\mathds{1}}+\delta S
\end{equation*}%
where $S$ is a $1$-Lipschitz map such that $S(0)=S(1)=0$ with support $\chi
_{S},$ finite union of intervals as in Proposition \ref{nearness}.

\item[A4] At each iteration of $T_{\delta }$, uniform random noise of radius 
$a>0,$ with noise kernel $\rho _{0}=a^{-1}1_{[-a/2,a/2]}$ is added, so that
the resulting transfer operator is given by 
\begin{equation*}
L_{0,T_{\delta }}f:=\rho _{0}\hat{\ast}L_{D_{\delta }\circ T_{0}}f
\end{equation*}%
as in Section \ref{sec:Main}.
\end{itemize}

\begin{remark}
We remark that piecewise expanding maps satisfy assimption $A1$ (see \cite[%
Section 3.1, Property E3]{Vi} or \cite{Li2} and in particular the Lasota
Yorke inequality satistied by these maps). The assumption is also verified
by maps which can be contracting in some part of the space.
\end{remark}

\begin{remark}
Assumption $A2$ \ {implies the} topological mixing of the map. In the case
of piecewise expanding maps this assumption is often taken to get a
topologically mixing system (see \cite[Section 3.1, Property E3]{Vi}). This
assumption is not the most general one possible to get linear response in
our framework, but it will keep the exposition simple.
\end{remark}

We will perturb both the map and the radius of the noise, applying \
Corollaries \ref{cormap2} and \ref{cornoise}. To apply it first we have to
check that the unperturbed operators we consider are mixing.

\begin{lemma}
Under the assumptions above, the transfer operator $L_{0,T_{0}}$ is mixing
i.e. it satisfies Assumption $(LR1)$ of Theorem \ref{th:linearresponse}.
\end{lemma}

\begin{proof}
If $L_{0}:=L_{0,T_{0}}$ is the transfer operator of the initial system with
additive noise, under the assumptions $A1$,...,$A4$, we have by $($\ref%
{conv3}$),$ as in Lemma \ref{lem:fixedpoint}%
\begin{eqnarray*}
||L_{0}^{n}f||_{BV} &\leq &\Vert L_{0}\Vert _{L^{1}\rightarrow BV}\Vert
L_{0}^{n-1}(f)\Vert _{1} \\
&\leq &0||f||_{BV}+9B\Vert f\Vert _{1}
\end{eqnarray*}%
where $B=||\rho _{0}||_{BV}$. Thus the system with noise satisfies a Lasota
Yorke inequality on $BV$ and $L^{1}$. Moreover $\Vert L_{0}f\Vert _{1}\leq
\Vert f\Vert _{1}$, and $L_{0}$ is compact as an operator $BV\rightarrow
L^{1},$ thus by the Hennion-Neussbaum argument (see \cite{Li2} Section 3.1,
Theorem 3.1) the spectral radius of $L_{0}$ as an operator on $BV$ is
bounded by $1$, the essential spectral radius is $0$. The spectrum is then
discrete, all points of the spectrum on the unit circle are eigenvalues with
finite multiplicity and the system is mixing if the only eigenvalue on the
unit circle is $1$ with multiplicity $1$.

Let us consider one positive stationary probability measure $\mu _{0}\in BV$
for $L_{0}$ (as was proved to exist in Lemma \ref{lem:fixedpoint}). \
Suppose that the system is not mixing, then there is a complex measure $\hat{%
\mu}\in BV$, $\hat{\mu}\neq \mu _{0}$ such that {$L_{0}^{i}\hat{\mu}=\lambda
^{i}\hat{\mu}{\ }$} for each $i\geq 0$ and {\ }$\lambda \in \mathbb{C}$, $%
|\lambda |=1$. Let $\mu $ be the real part of $\hat{\mu}.$ This is a signed
measure.  Since the transfer operator preserves real valued measures $%
L_{0}^{i}\mu $ is a real valued measure with bounded variation density and,
for each $\epsilon $ there are infinitely many $i$ such that {\ $\Vert
L_{0}^{i}\mu -\mu \Vert _{1}\leq \epsilon $ \ (}for each $\epsilon $, $%
|\lambda ^{i}-1|\leq \epsilon $ for infinitely many $i$). We also have that
there is $c\in \mathbb{R}$ such that $\mu _{1}=\mu +c\mu _{0}$ is a zero
average measure with density in $BV$ and 
\begin{equation*}
\underset{i\rightarrow \infty }{\lim \sup }{\Vert L_{0}^{i}\mu _{1}\Vert
_{1}=\Vert \mu _{1}\Vert _{1}}
\end{equation*}%
(indeed ${\Vert L_{0}^{i}\mu _{1}\Vert _{1}\leq \Vert \mu _{1}\Vert _{1}}$ $%
\ $and ${\Vert L_{0}^{i}\mu _{1}\Vert _{1}=||L_{0}^{i}(}\mu +c\mu
_{0})||_{1}=||{L_{0}^{i}(}\mu )-\mu +\mu +{L_{0}^{i}(}c\mu _{0})||_{1}\leq ||%
{L_{0}^{i}(}\mu )-\mu ||_{1}+{\Vert \mu _{1}\Vert _{1}}$ ).

Now let $I$ \ be an interval for which $\mu _{1}|_{I}$ has a strictly
positive density. By assumption $A2$ there is $n$ such that $%
T_{0}^{n}(I)=[0,1]$. Let us consider the measure $\nu =\mu _{1}1_{I}$.
Suppose $s(\nu )$ is the support of $\nu $ (the set on which $\nu $ has
strictly positive density). Since $L_{T_{0}}$ is a positive operator and  $%
T_{0}$ is piecewise $C^{2}$, then $L_{T_{0}}\nu $ has also a strictly
positive density almost everywhere on $T_{0}(I).$ Since the convolution can
only increase the support of a positive measure, we get that $s(L_{0}\nu
)\supseteq \overline{s(L_{T_{0}}\nu )}$ $\supseteq T_{0}(I)$,  $%
s(L_{0}^{2}\nu )\supseteq T_{0}^{2}(I)$ and $s(L_{0}^{i}\nu )\supseteq
T_{0}^{i}(I)$ for $i\geq 1$. Then by assumption $A2$ we get inductively $%
s(L_{0}^{n}\nu )=[0,1]$. \ This contraddicts the fact that $\underset{%
i\rightarrow \infty }{\limsup}{\Vert L_{0}^{i}\mu _{1}\Vert _{1}}${$%
=\Vert \mu _{1}\Vert _{1}$}. Indeed \ {recall that any measure }$\mu _{1}${\
of zero average can be decomposed in $\mu _{1}^{+}+\mu _{1}^{-}$, the
positive and negative component of }$\mu _{1}${. }We have that $L_{0}^{n}\mu
_{1}^{-}$ is a negative measure having a bounded variation density and the
support of $L_{0}^{n}\nu $ being the whole space overlaps the support of $%
L_{0}^{n}\mu _{1}^{-}$ in this way 
\begin{eqnarray*}
{\Vert L_{0}^{n}\mu }_{1}{\Vert _{1}} &=&\Vert L_{0}^{n}(\mu _{1}^{+}+\mu
_{1}^{-})\Vert _{1} \\
&\leq &\Vert L_{0}^{n}(\mu _{1}^{+}-\nu +\nu +\mu _{1}^{-})\Vert _{1} \\
&\leq &\Vert L_{0}^{n}(\mu _{1}^{+}-\nu )||_{1}+||L_{0}^{n}\nu +L_{0}^{n}\mu
_{1}^{-}\Vert _{1} \\
&<&\Vert L_{0}^{n}(\mu _{1}^{+}-\nu )||_{1}+||\mu _{1}^{-}||_{1}+||\nu ||_{1}
\\
&=&||\mu _{1}||_{1}.
\end{eqnarray*}

Then for each $k\geq 0$ $\ {\Vert L_{0}^{n+k}\mu }_{1}{\Vert _{1}\leq \Vert
L_{0}^{n}\mu }_{1}{\Vert _{1}<}||\mu _{1}||_{1}$, contraddicting $\underset{%
i\rightarrow \infty }{\lim \sup }{\Vert L_{0}^{i}\mu _{1}\Vert _{1}}${$%
=\Vert \mu _{1}\Vert _{1}$}.
\end{proof}

\begin{remark}
The assumptions A1,...,A4 have been chosen to ensure mixing for a whole
family of nontrivial systems. Given a single system with noise, to prove its
mixing is usually much easier. It can be done with a computer aided proof
also in quite complicated examples, as we will see in Section \ref{sec:final}%
.
\end{remark}

Applying Corollaries \ref{cormap2} and \ref{cornoise}, using the computation
made in Example \ref{exnoise} for the explicit form of the derivative
operator we get

\begin{corollary}
\label{cor:tentspertubartion} Using the notations of Corollary \ref{cormap2}%
, and \ref{cornoise}, for the piecewise expanding systems with uniform noise
described above, perturbing the deterministic part of the system as in
Assumption A3 \ we get the following Linear Response formula%
\begin{equation}
\lim_{\delta \rightarrow 0}\left\Vert \frac{f_{\delta }-f_{0}}{\delta }%
-R(1,L_{0})\rho _{0}\hat{\ast}(-L_{T_{0}}(f_{0})S)^{\prime }\right\Vert
_{1}=0.
\end{equation}

While if the map $T_{0}$ is Lipschitz and we perturb the uniform noise
kernel by changing the radius we get he following Linear Response formula%
\begin{equation*}
\lim_{\delta \rightarrow 0}\left\Vert \frac{f_{\delta }-f_{0}}{\delta }%
-R(1,L_{0})\dot{\rho}\hat{\ast}L_{T}(f)\right\Vert _{W}=0.
\end{equation*}%
where%
\begin{equation}
\dot{\rho}=-a^{-2}1_{[-a/2,a/2]}+\frac{a^{-1}}{2}\delta _{-a/2}+\frac{a^{-1}%
}{2}\delta _{a/2}
\end{equation}%
as in $(\ref{rodot})$.
\end{corollary}

\begin{remark}
We remark that if $L_{T}:BV[0,1]\rightarrow BV[0,1]$ as required in
Assumption $A1$ is continuous thus $L_{T}:BV([0,1])\rightarrow BV(\chi _{S})$
is automatically continuous (see Remark \ref{bvcase}). This allows to
perturb the map in a way that moves discontinuities and still get Linear
Response. This is of course due to the effect of the additive noise. Such
perturbations may break linear response or statistical stability in the
deterministic case (see e.g. \cite{Ba1} and \cite{Maz}).
\end{remark}

%\textcolor{red}{\bf 
%An analogous reasoning, using the tools of Section  \ref{secptn}, proves linear response for perturbations of the noise i.e. changing the %shape of the noise kernel.}

%\textcolor{blue}{\ifmmode\t%ext{\sout{\ensuremath{Moreover, other perturbations, like changing the shape of the kernel also
%give a linear response in this case (see Section \ref{secptn} for the tools).}}}
%\else\sout{Moreover, other perturbations, like changing the shape of the kernel also
%give a linear response in this case (see Section \ref{secptn} for the tools).}\fi }%

\subsection{Random rotations}

\label{slat} Rotations by a well approximable angle are known (see \cite%
{Gpre} e.g.) to be systems having not linear response even to constant
deterministic perturbations. In this section, we show that, not
surprisingly, if we add some noise they have.

Let us hence consider $([0,1],T)$ being such a translation, with $%
T(x)=x+\theta $ ${mod}(1)$. We will suppose that $\theta $ is of large
Diophantine type. Let us recall the definition of Diophantine type. The
definition tests the possibility of approximating $0$ by an integer
multiples of the angle. The notation $\left\vert \left\vert \cdot
\right\vert \right\vert $ below, will indicate the distance to the nearest
integer number in $\mathbb{R}$.

\begin{definition}
\label{linapp} The Diophantine type of $\theta $ is 
\begin{equation*}
\gamma (\theta )=\inf \{\gamma \, | \, \exists c_{0}>0~s.t.\,\Vert k\theta
\Vert \geq c_{0}|k|^{-\gamma }~\forall k\in \mathbb{Z}\setminus\{0\} \}.
\end{equation*}
\end{definition}

Rotations by a well approximable angle do not have linear response to
certain small Lipschitz perturbations as it is shown in the following
propositions (see \cite{Gpre} section 6.3 for the proofs).

\begin{proposition}
\label{bahh}Consider a well approximable irrational $\theta $ with $\gamma
(\theta )>2$, consider $\gamma ^{\prime }>\gamma (\theta )$. There is a
sequence of reals $\delta _{j}\geq 0$, $\delta _{j}\rightarrow 0$ and
Lipschitz small perturbations\footnote{%
Informally speaking these perturbations chage the rotation angle to a nearby
rational, making an invariant measure supported on periodic orbit to appear,
then this periodic orbit is made attracting by a further small perturbation
of the map.} $T_{\delta _{j}}$such that $\ \Vert T-T_{\delta _{j}}\Vert
_{Lip}\leq 2\delta _{j}$ and the map $T_{\delta _{j}}$ has an unique
invariant measure $\mu _{j}$ with%
\begin{equation*}
\Vert \mu _{0}-\mu _{j}\Vert _{W}\geq \frac{1}{9}(\delta _{j})^{\frac{1}{%
\gamma ^{\prime }-1}}.
\end{equation*}
\end{proposition}

The result also hold for the average of a given regular observable. In \cite%
{Gpre} it is shown an explicit example of such an observable for a rotation
of a well-chosen angle $\theta$.

\begin{proposition}
Consider a map $T$ as above with the rotation angle $\theta
=\sum_{1}^{\infty }2^{-2^{2i}}$. There is a sequence of reals $\delta
_{j}\geq 0$, $\delta _{j}\rightarrow 0$ and Lipschitz small perturbations $\
\Vert T-T_{\delta _{j}}\Vert _{Lip}\leq 2\delta _{j}$ with invariant
measures $\mu _{j}$, an observable $\psi :[0,1]\rightarrow \mathbb{R}$ \
with derivative in $L^{2}$ (see \cite{Gpre} section 6.3 for its definition)
and $C\geq 0$ such that 
\begin{equation*}
\left\vert \int \psi dm-\int \psi d\mu _{j}\right\vert \geq C\sqrt{\delta
_{j}}.
\end{equation*}
\end{proposition}

Now consider a random system, where to a rotation as above we add noise with
BV density of probability $\rho _{\xi }$. The related annealed transfer
operator can be defined as:%
\begin{equation}
L_{\xi ,T}(f):=\rho _{\xi }\ast L_{T}(f)
\end{equation}%
here the convolution is taken onto the circle, or equivalently on the
interval with periodic boundary conditions\footnote{%
Formally, it can be defined as in (\ref{hat}) considering as $\pi $ the
universal cover map $\mathbb{R\rightarrow }S^{1}$.}, as it is natural for
rotations. Let us consider again some small Lipschitz perturbations of $T$
as done before where we suppose $T_{\delta }=$ $D_{\delta }\circ T_{0}$ with 
$D_{\delta }={\mathds{1}}+\delta S$, a bilipschitz homeomorphism near the
identity and the related transfer operators $L_{\xi ,T_{\delta }}$. We
remark that for the system $L_{\xi ,T_{0}}$, the Lebesgue measure $m$ is the
stationary measure.\ We can apply Corollary \ref{cormap2} provided the
unperturbed operator is mixing. About this, let us notice that in the case
of rotations the convolution and $L_{T}$ commute, thus%
\begin{equation*}
(L_{\xi ,T}(f))^{n}=\underset{\llcorner ~n~times~\lrcorner }{\rho _{\xi
}\ast ...\ast \rho _{\xi }}\ast L_{T}^{n}(f)
\end{equation*}%
there is hence then some $n$ \ for which $(L_{\xi ,T})^{n}$ has a strictly
positive kernel and it is mixing (see Remark \ref{rmkLM}) and we get the
linear response formula converging in $L^{1}$. 
\begin{equation*}
\lim_{\delta \rightarrow 0}\frac{f_{\delta }-m}{\delta }=(1-L_{\xi
,T_{0}})^{-1}\rho _{\xi }\ast (-S)^{\prime }.
\end{equation*}

\subsection{A model of the Belosuv-Zhabotinsky reaction}

\label{sec:final}

We show an example of a random system of interest in applications composed
by a nontrivial map perturbed by noise. The deterministic part $T$ of the
system has \ coexisting strong expansion and contraction regions, making the
mathematical understanding of the statistical properties of the system,
quite difficult. However, it is possible, with a computer aided proof to
show that this system with additive noise satisfies Assumption $(LR1)$ of
Theorem \ref{th:linearresponse} (mixing). The other needed assumptions can
be verified directly and then this system has a linear response under
suitable small perturbations of $T$.

The system we consider is a model of the behavior of the famous
Belosuv-Zabotinsky chaotic chemical reaction (see \cite{MT}). This is a
dynamical system with additive noise in which we show response for certain
perturbations of the map. The deterministic part of the system is given by
the map 
\begin{equation}
T(x)=\left\{ 
\begin{array}{c}
(a+(x-\frac{1}{8})^{\frac{1}{3}})e^{-x}+b,~~~0\leq x\leq 0.3 \\ 
c(10xe^{\frac{-10x}{3}})^{19}+b~~~0.3\leq x\leq 1%
\end{array}%
\right.  \label{BZ}
\end{equation}

Where the parameter $c$ is defined so that $T(0.3^{-})=T(0.3^{+})$, making $%
T $ continuous at $0.3$. The value of $c$ can be computed in a closed form
as: 
\begin{equation*}
c=\frac{20}{3^{20}\cdot 7}\cdot \bigg(\frac{7}{5}\bigg)^{1/3}\cdot
e^{187/10}.
\end{equation*}

The parameter $a$ is defined so that so that $T^{\prime }(0.3^{-})=0$,
making $T^{\prime }$ continuous at $0.3$. The value of $a$ can be computed
in a closed form as: 
\begin{equation*}
a=\frac{19}{42}\cdot \bigg(\frac{7}{5}\bigg)^{1/3}
\end{equation*}%
while 
\begin{equation*}
b=0.02328852830307032054478158044023918735669943648088852646123182739831022528
\end{equation*}%
is near to a value giving a period $3$ orbit for the critical value $T(0.3)$
(see \cite{GMN} for more details). 
\begin{figure}[tbph]
\centering
\includegraphics[height=4cm, width=5cm]{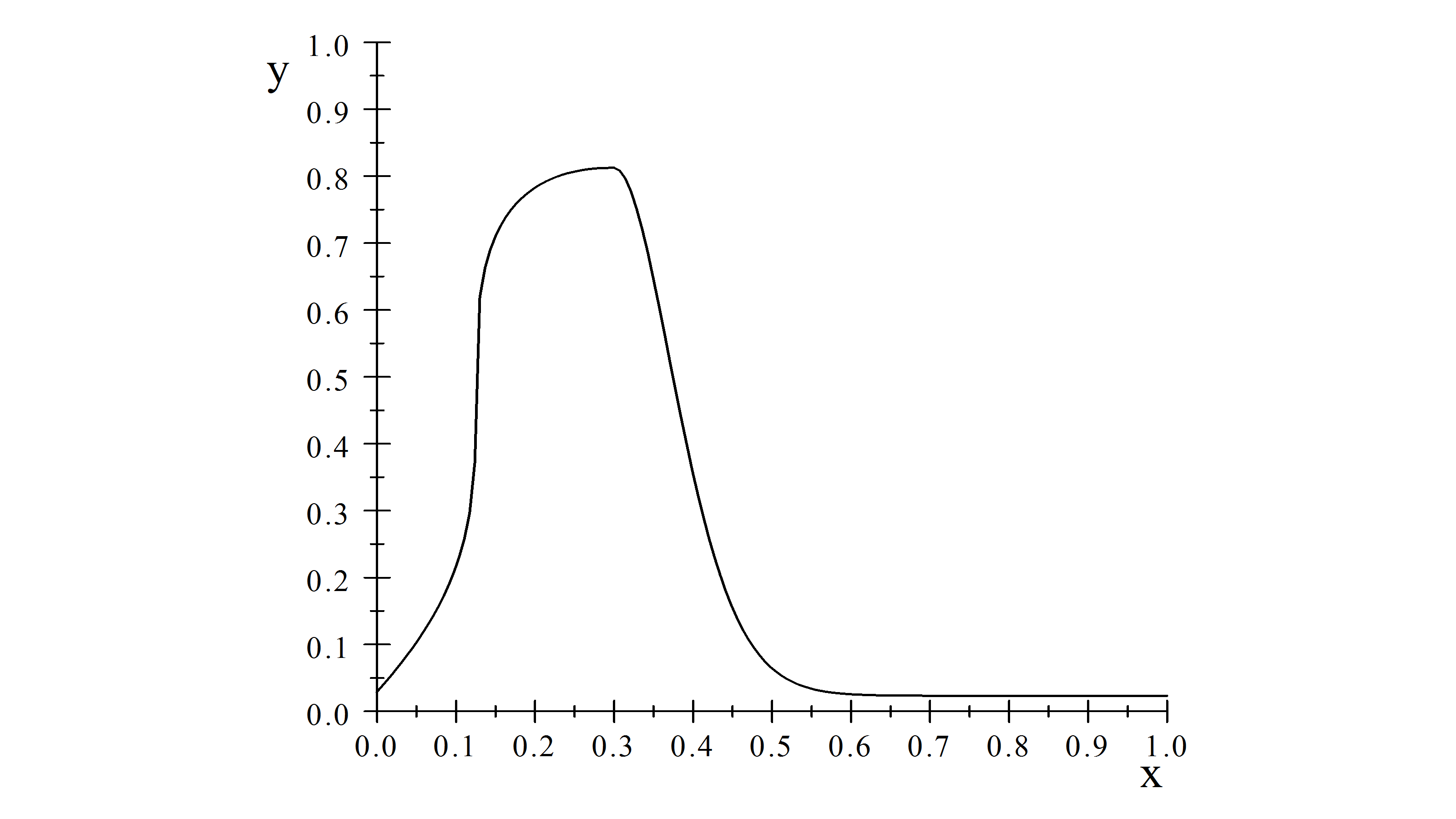}
\caption{The map $T$.}
\label{f0}
\end{figure}
The graph ot the map $T$ is represented in Figure \ref{f0}. At each
iteration of the map a uniformly distributed noise perturbation with span of
size $\xi $ is applied, by a kernel $\rho _{0}=\xi ^{-1}1_{[-\xi /2,\xi /2]}$%
, We consider reflecting boundaries when the noise is big enough to send
points out of the space $[0,1]$. We remark that this is never the case when
the noise amplitude is smaller than the parameter $b$.

In \cite{MT} the authors consider several different values for the noise
range, showing numerically that there is a transition from positive to
negative Lyapunov exponents as the noise range increases. This statement was
proved with a computer aided proof in \cite{GMN}. We now select two
different noises amplitudes (one with positive and one with negative
Lyapunov exponent) for which the mixing assumption can be proved.

\begin{lemma}
\label{mixa}Let us consider $\xi _{1}=$ $0.860\times 10^{-2}$ and $\xi
_{2}=0.129\times 10^{-3}$ and the kernels $\rho _{i}=\xi _{i}^{-1}1_{[-\xi
_{i}/2,\xi _{i}/2]}$ \ Let us consider the transfer operator $L_{\xi
_{i}}:=\rho _{i}\hat{\ast}L_{T}$\ associated to these systems satisfy
assumption $(LR1)$ of Theorem \ref{th:linearresponse} (mixing).
\end{lemma}

\begin{proof}
In \cite{GMN}, Table 1 (first and last line) the reader can find a rigorous
estimate for the mixing rate of the transfer operator\ $L_{\xi _{i}}$
related to the size noise $\xi _{i}.$\ In particular it is shown that for
the noise size $\xi _{1}$ \ \ $\Vert L_{\xi _{1}}^{55}|_{V}\Vert
_{L^{1}}\leq 0.059$ \ and $\xi _{2}$ \ \ $\Vert L_{\xi _{2}}^{70}|_{V}\Vert
_{L^{1}}\leq 0.41$ \ proving the exponential contraction of the zero average
space in $L^{1}$ and in particular the mixing assumption.
\end{proof}

{Note that with those values of the noise, the transfer operator is
regularizing from $L^{1}$ to $BV[0,1]$ in the sense of Corollary \ref{regcor}%
. }The above statement also holds for the other noise amplitudes considered
in \cite{GMN}, but we will not enter in the details in this paper. Now we
focus on a class of perturbations of $T$ such that the derivative operator
exist. Let us consider family of perturbations $T_{\delta }=$ $D_{\delta
}\circ T$ as in the previous section, such that $D_{\delta }={\mathds{1}}%
+\delta S$ with $S$ 1-Lipschitz, suppose that the support $\chi _{S}$ of $S$
does not contain a neighborhood of the global maximum $m_{T}$ of $T.$ We see
that such a class of perturbations satisfies Proposition \ref{pertmap}.

\begin{lemma}
\label{buco}Let $T$ be the map defined at $($\ref{BZ}$)$. Let us consider $%
T_{\delta }=$ $D_{\delta }\circ T$, such that $D_{\delta }={\mathds{1}}%
+\delta S$ and $\chi _{S}$ does not contain a neighborhood of $m_{T}$. Then $%
L_{T}:BV([0,1])\rightarrow BV(\chi _{S})$ is continuous.
\end{lemma}

\begin{proof}
This follows from the remark that outside a neighborhood of $m_{T}$ the map
is made of a finite set of convex monotone $C^{2}$ branches with derivative
bounded away from zero.
\end{proof}

We now have everything we need to apply Theorem \ref{th:linearresponse}
(Corollary \ref{cormap2}) and get

\begin{proposition}
\label{BZappl}Let us consider the random dynamical system \ formed by the
map $T$ defined at (\ref{BZ}) and uniformly distributed additive noise of
radius $\xi _{i}$ as in the statement of Lemma \ref{mixa}. \ If we consider
a family of systems obtained by the deterministic perturbation of $T$ as
described before: $T_{\delta }=$ $D_{\delta }\circ T$, such that $D_{\delta
}={\mathds{1}}+\delta S$, $S$ being $1$-Lipschitz and $\chi _{S}$ does not
contain a neighborhood of $m_{T}$. Let us consider the associated transfer
operators $L_{\xi ,T_{\delta }}:=\rho _{\xi }\hat{\ast}L_{D_{\delta }\circ
T}.$ Then the system has linear response: let $f_{0}$ be the stationary
measure of $L_{\xi ,T_{0}}$ and $f_{\delta }$ some stationary measure of $%
L_{\xi ,T_{\delta }},$ then 
\begin{equation*}
\lim_{\delta \rightarrow 0}\left\Vert \frac{f_{\delta }-f_{0}}{\delta }%
-R(1,L_{\xi ,T_{0}})[\rho _{\xi }\hat{\ast}(-L_{T}(f_{0})S)^{\prime
}]\right\Vert _{1}=0.
\end{equation*}
\end{proposition}

\subsection{Small random deterministic perturbations of a system with
additive noise}

\label{detappl}

In this subsection we consider an example of a family of random systems
modeling the following situation: we consider two systems, one is
deterministic and the other is random with additive noise. At every
iteration we can randomly decide to apply one system or the other. The
random system is a random map with additive noise {\ $x_{n+1} \rightarrow
T_{1}(x_{n})+\omega_n $} where $T_{1}$ is nonsingular and {\ $\omega _{n}$
is an i.i.d random variable, as in equation \eqref{systm}}, distributed
according to a certain $BV$ kernel $\rho $ (with associated transfer
operator $L_{0}$ as described in Section \ref{sec:Main}). We apply this
system with probability $(1-\delta ).$ The deterministic system is a\emph{\ }%
nonsingular\emph{\ }map $T_{2}$ (the transfer operator associated to $T_{2}$
is denoted by $L_{T_{2}}:L^{1}\rightarrow L^{1}$) we apply this second
system with probability $\delta $ independent from the previous history of
the system. We will also suppose that $T_{2}$ is such that $%
L_{T_{2}}:BV[0,1]\rightarrow BV[0,1]$ is continuous. When $\delta $ is small
we consider this as a small random perturbation of the random system
described by $L_{0}$. The (annealed) transfer operator associated to the
randomly perturbed system can be defined by%
\begin{equation}
L_{\delta }=(1-\delta )L_{0}+\delta L_{T_{2}}.  \label{defino}
\end{equation}%
We remark that for $\delta >0$ this operator has not an absolutely
continuous kernel; however our theory can still be applied to it provided $%
L_{0}$ and its perturbation satisfy the assumptions of Theorem \ref%
{th:linearresponse}. {\ Systems of the type \eqref{defino} were already
studied in \cite[Section 10.4]{LM94}, called there \textquotedblleft
randomly applied stochastic perturbation\textquotedblright , with respect to
asymptotic stability (see the reference for the details) but not with
respect to linear response. See also the beginning of \cite[Section 10.5]{LM94} for other relevant cases of similar situations and a bit of
literature. In this case we have the following situation
\begin{proposition}
Let $L_{0}$ and $L_{T_{2}}$ as above, suppose $L_{0}$ satisties Assumption $%
(LR1)$ of Theorem \ref{th:linearresponse}\footnote{%
A nontrivial example of such a system can be the BZ system considered in
Secion \ref{sec:final}.}. Let us consider the perturbation $L_{\delta }$ of $%
L_{0}$ as defined in $(\ref{defino})$. Suppose $L_{\delta }f_{\delta
}=f_{\delta }$ for $\delta \in \left[ 0,\overline{\delta }\right) $. Then
the following response formula holds%
\begin{equation*}
\lim_{\delta \rightarrow 0}\left\Vert \frac{f_{\delta }-f_{0}}{\delta }%
-R(1,L_{0})(L_{T_{2}}f_{0}-f_{0})\right\Vert _{1}=0.
\end{equation*}
\end{proposition}
\begin{proof}
The statement is a direct application of Theorem \ref{th:linearresponse}.
Assumption $(LR1)$ is supposed to hold for $L_{0}$. As before, Assumption $%
(LR2)$ of Theorem \ref{th:linearresponse} is satisfied because of the
regularization inequalities proved in Corollary \ref{regcor}. We verify that $\|f_{\delta }\|_{BV}$ is uniformly bounded (assumption $(LR0)$ of Theorem \ref{th:linearresponse}). Here we use that $\|L_{T_{2}}\|_{BV\rightarrow BV}$
is bounded. Remark that 
\begin{equation*}
||L_{\delta }\mu ||_{BV}=||(1-\delta )L_{0}\mu ||_{BV}+||\delta L_{T_{2}}\mu
||_{BV}\leq 9||\rho ||_{BV}+\delta ||L_{T_{2}}||_{BV\rightarrow BV}||\mu
||_{BV}.
\end{equation*}%
If $\delta $ is small ehough \ $\delta ||L_{T_{2}}||_{BV\rightarrow BV}<1$
and this is a kind or Lasota Yorke inequality. Iterating it, we get that
eventually for $n$ large enough%
\begin{equation*}
||L^{n}\mu ||_{BV}\leq \frac{2||\rho ||_{BV}}{1-\delta
||L_{T_{2}}||_{BV\rightarrow BV}}
\end{equation*}%
and we can perform the same construction as in Lemma \ref{extuniq}. To
verify Assumption $(LR3)$ we consider as a weak norm the $L^{1}$ norm
itself. To verify the first part of Assumption $(LR3)$ it is sufficient to
prove that there is $K>0$ such that%
\begin{equation*}
\left\vert |L_{0}-L_{\delta }|\right\vert _{L^{1}\rightarrow L^{1}}\leq
K\delta
\end{equation*}%
for $\delta $ small enough. This is done by noticing that in our case 
\begin{eqnarray*}
||L_{0}-L_{\delta }||_{L^{1}\rightarrow L^{1}} &=&||\delta L_{T_{2}}-\delta
L_{0}||_{L^{1}\rightarrow L^{1}}=\delta
||L_{T_{2}}-L_{0}||_{L^{1}\rightarrow L^{1}} \\
&\leq &2\delta .
\end{eqnarray*}%
(the transfer operators are weak contractions in $L^{1}$). To verify the
second part of Assumption $(LR3)$ and compute a derivative operator\ we have
to show the convergence in the $L^{1}$ of the limit%
\begin{equation*}
\underset{\delta \rightarrow 0}{\lim }\frac{(L_{0}-L_{\delta })}{\delta }%
f_{0}.
\end{equation*}%
Again this is very simple because 
\begin{equation*}
\underset{\delta \rightarrow 0}{\lim }\frac{(L_{0}-L_{\delta })}{\delta }%
f_{0}=(L_{T_{2}}-L_{0})f_{0}=L_{T_{2}}f_{0}-f_{0}.
\end{equation*}%
Applying Theorem \ref{th:linearresponse} we get a linear reponse formula: 
\begin{equation}
\lim_{\delta \rightarrow 0}\left\Vert \frac{f_{\delta }-f_{0}}{\delta }%
-R(1,L_{0})(L_{T_{2}}f_{0}-f_{0})\right\Vert _{1}=0.
\end{equation}
\end{proof}
\section{(Optimal) control of the statistical properties}
\label{sec:control}
An important problem related to linear response is the control of the
statistical properties of a system: how one can perturb the system, in order
to modify its statistical properties in a prescribed way? how can one do it
optimally?\ (what is the best action to be taken in a possible set of
allowed small perturbations in order to achieve a given small modification
of the stationary measure?).

The understanding of this problem has potentially a great importance in the
applications, as it is related to questions about optimal strategies in
order to influence the behavior of a system. As an example, thinking about
climate models one could ask \textquotedblleft what is the best action to be
taken in order to reduce the average temperature?\textquotedblright\ or
similarly for other statistical properties. This type of questions is still
not much investigated in the literature. In \cite{GP} the problem was faced
and discussed in general, giving a detailed description of the solutions and
the existence of optimal ones in the case of deterministic perturbations of
expanding maps. In \cite{Kl16} a similar problem was considered for more
general systems, restricting the allowed set of perturbations to
conjugacies. In \cite{Mac} problems of this kind were investigated in
connection with the management of complex dynamical systems (networks with
many interdependent components). In \cite{ADF} several related problems,
still focusing on optimal perturbations, are investigated for Markov Chains
and applied to Ulam approximations of random dynamical systems on the
interval.

Let us start formalizing the problem more precisely: suppose we have a
system represented by its transfer operator $L_{0}$ and a family of
perturbations $L_{\delta ,\gamma },$ of magnitude $\delta $ and direction $%
\gamma \in D$ varying in a set $D$ of the allowed ``infinitesimal''
perturbations. Suppose $f_{0}$ is a stationary measure of $L_{0}$ and $%
f_{\delta ,\gamma }$ is a stationary measure of $L_{\delta ,\gamma }$

\begin{enumerate}
\item can we find a perturbation $\gamma \in D,$ leading to some wanted
direction of change of the stationary measure? (in the sense of prescribed
linear response $\mu $).

\item In case of many solutions in $D$, can we find an optimal one?

\item In case of no solutions in $D$, can we find an optimal perturbation in 
$D$, approximating as well as possible the wanted response?
\end{enumerate}

Formally the request problem $1)$ translates into the following: given $\mu
, $ find $\gamma $ such that 
\begin{equation*}
\lim_{\delta \rightarrow 0}\frac{f_{\delta ,\gamma }-f_{0}}{\delta }=\mu .
\end{equation*}%
The limit above can be considered in different topologies, in this paper we
considered the $L^{1}$ and the $W$ topology, but other topologies could be
considered, including the convergence under different observables: i.e.
suppose $\psi $ is a smooth observable with values in $\mathbb{R}^{n}$, one
may consider the problem of finding $\gamma $ for which it holds 
\begin{equation*}
\lim_{\delta \rightarrow 0}\frac{\int \psi ~df_{\delta ,\gamma }-\int \psi
~df_{0}}{\delta }=\int \psi ~d\mu
\end{equation*}%
in $\mathbb{R}^{n}$. Formalizations of problems $2)$ and $3)$ can be made
similarly, using linear response. Similar questions may consider the maximum
response in some norm or the one of a given observable in some norm (see 
\cite{ADF}).

Having obtained in the previous section handy and explicit results for the
linear response of systems with additive noise, we now consider problem $1)$
and discuss briefly its mathematical structure. As done in the previous
sections, let us allow perturbations of the deterministic part of the system
of the form $T_{\delta }=({\mathds{1}}+\delta S)\circ T$, with $S$ being $1$%
-Lipschitz and let us consider as perturbed operators $L_{\delta ,\gamma }$
the associated transfer operators $L_{\delta ,S}:=\rho _{\xi }\ast L_{({%
\mathds{1}}+\delta S)\circ T}.$ Consider the linear response formula \ found
before for these kind of perturbations: let $f_{0}$ be the stationary
measure of $L_{0,S}$ and $f_{\delta }$ some stationary measure of $L_{\delta
,S},$ then under the assumptions of Theorem \ref{th:linearresponse} and
Proposition \ref{derivop} we have the linear response formula, 
\begin{equation*}
\lim_{\delta \rightarrow 0}\frac{f_{\delta }-f_{0}}{\delta }=R(1,L_{\xi
,T_{0}})[\rho _{\xi }\hat{\ast}(-L_{T_{0}}(f_{0})S)^{\prime }]=\mu
\end{equation*}%
(with convergence in $L^{1}$). This equation in now to be solved for $S$.
Leading to 
\begin{equation*}
\rho _{\xi }\hat{\ast}(-L_{T_{0}}(f_{0})S)^{\prime }=\mu {-}L_{\xi
,T_{0}}\mu .
\end{equation*}%
Now denote by $N_{\xi }:L^{1}\rightarrow BV$ the convolution operator: $%
N_{\xi }(f)=\rho _{\xi }\hat{\ast}f$ . $N_{\xi }$ is not necessarily
injective or onto (it is not injective due to the boundary arrangements, but
in the case of the map defined at $($\ref{BZ}$)$ it is when the noise is
small enough because the image of the map is strictly contained in $[0,1]$).
Suppose $\mu {-}L_{\xi ,T_{0}}\mu $ is in the range of $N_{\xi }$ and denote
by 
\begin{equation*}
N_{\xi }^{-1}(g)=\{f\in L^{1}~s.t.N_{\xi }(f)=g\}.
\end{equation*}%
We have%
\begin{equation*}
(-L_{T}(f_{0})S)^{\prime }\in N_{\xi }^{-1}(\mu {-}L_{\xi ,T_{0}}\mu )
\end{equation*}%
leading to the explicit family of solutions to the control problem $1)$ 
\begin{equation*}
S(t)=\frac{C+\int_{0}^{t}f~dm}{-L_{T_{0}}(f_{0})}
\end{equation*}%
for every $C\in \mathbb{R}$ and $f\in N_{\xi }^{-1}(\mu {-}L_{\xi ,T_{0}}\mu
)$ when the expression makes sense. Inside this family one can search for
optimal or optimal approximating solutions, as proposed at problems $2)$ and 
$3)$ above. Further investigations of these problems are in our opinion very
interesting, but out of the scope of the present paper.

\begin{remark}
Note that, analogously to what has been done before in remark \ref{rem_GP17}%
, we can compare this result with the one in \cite{GP}. Over there the
solution to the problem is computed, given $\rho_1$ by the solution of a
differential equation which is, morally, what has been done here.
\end{remark}

\end{document}